\documentclass[12 pt]{amsart}
\usepackage{amssymb,amsmath,amsfonts,amsthm, mathrsfs,tikz, pgfplots,multicol, graphicx,stmaryrd,placeins,hyperref}
\usepackage[margin=1in]{geometry}
\newtheorem{lemma}{Lemma}
\newtheorem{theorem}{Theorem}
\newtheorem{cor}{Corollary}
\newtheorem{prop}{Proposition}

\newcommand{\mbb}[1]{\mathbb{#1}}
\newcommand{\mc}[1]{\mathcal{#1}}
\newcommand{\floor}[1]{\left\lfloor #1\right\rceil}
\newcommand{\rem}[1]{\left\langle#1\right\rangle}
\begin{document}
\title{Badly approximable numbers over imaginary quadratic fields}
\author[Robert Hines]{Robert Hines}
\address{
Department of Mathematics, University of Colorado,
Campus Box 395, Boulder, Colorado 80309-0395}
\email{robert.hines@colorado.edu}

\begin{abstract}
We recall the notion of nearest integer continued fractions over the Euclidean imaginary quadratic fields $K$ and characterize the ``badly approximable'' numbers, ($z$ such that there is a $C=C(z)>0$ with $|z-p/q|\geq C/|q|^2$ for all $p/q\in K$), by boundedness of the partial quotients in the continued fraction expansion of $z$.  Applying this algorithm to ``tagged'' indefinite integral binary Hermitian forms demonstrates the existence of entire circles in $\mbb{C}$ whose points are badly approximable over $K$, with effective constants.

By other methods (the Dani correspondence), we prove the existence of circles of badly approximable numbers over \textit{any} imaginary quadratic field.  Among these badly approximable numbers are algebraic numbers of every even degree over $\mbb{Q}$, which we characterize.  All of the examples we consider are associated with cocompact Fuchsian subgroups of the Bianchi groups $SL_2(\mc{O})$, where $\mc{O}$ is the ring of integers in an imaginary quadratic field.
\end{abstract}
\maketitle

\section*{Introduction}


A natural generalization of continued fractions to complex numbers over appropriate discrete subrings $\mc{O}$ of $\mbb{C}$, in particular over $\mbb{Z}[\sqrt{-1}]$ and $\mbb{Z}\left[\frac{1+\sqrt{-3}}{2}\right]$, was introduced by A. Hurwitz, \cite{Hur1}.  Let $K$ be one of the Euclidean imaginary quadratic fields and $\mc{O}$ its ring of integers.  We write a complex number uniquely as $\floor{z}+\rem{z}$ with $\floor{z}\in\mc{O}$ the nearest integer to $z$ and $\rem{z}\in V$, where $V$ is the collection of complex numbers closer to zero than to any other point of the lattice $\mc{O}$ (with some choice along the boundary of $V$).  For $z\in V$ we iterate the map $T(z)=\rem{1/z}$, $T^n(z)=:z_n$, to obtain the continued fraction
$$
z=[a_0;a_1,a_2,\ldots]=a_0+\frac{1}{a_1+\frac{1}{a_2+\dots}}, \ \floor{1/z_i}=a_{i+1}\in\mc{O},
$$
and convergents $p_n/q_n=[a_0;a_1,\ldots,a_n]$.

It is known that the convergents $p_n/q_n$ from the above algorithms all satisfy
$$
|z-p_n/q_n|\leq C/|q_n|^2
$$
for some $C>0$.  See Proposition \ref{prop:approx_const} below for a proof and Theorem 1 of \cite{La} for the smallest values of $C$.  This is a simple algorithmic realization of Dirichlet's theorem that for irrational $z\in\mbb{C}$, there are infinitely many $p/q\in K$ satisfying the above inequality.  A number $z$ is \textit{badly approximable} if the exponent of two is the best possible, i.e. $z$ is badly approximable if there exists $C'>0$ such that for any $p/q\in K$ we have
$$
|z-p/q|\geq C'/|q|^2.
$$
It is well-known that a real number is badly approximable over $\mbb{Q}$ if and only if its partial quotients $a_n$ are bounded.  We show below (Theorem \ref{thm:bndpartquots}) that this is the case for nearest integer continued fractions over $K$ as well, relying on the work of Lakein \cite{La} who investigated the quality of approximation of the nearest integer convergents.

It is a folklore conjecture that the only real algebraic numbers with bounded partial quotients in their continued fraction expansion are the quadratic irrationals, whose partial quotients are eventually periodic.  However, it is shown in \cite{BG}, using methods expanded upon in this paper, that the analogous conjecture does not hold exactly over $\mbb{Q}(\sqrt{-1})$.  There are examples of algebraic numbers of relative degree greater than two over $\mbb{Q}(\sqrt{-1})$ whose nearest integer continued fraction expansions have bounded partial quotients.  Examples of this phenomenon were first detailed by Hensely, cf. \cite{Hen} $\S 5.6$.  While these examples are not quadratic over $\mbb{Q}(\sqrt{-1})$, they are associated with closed geodesic surfaces in the Bianchi orbifold $SL_2(\mbb{Z}[i])\backslash\mbb{H}^3$ in the same way that real quadratic irrationals are associated to closed geodesics on the modular surface $SL_2(\mbb{Z})\backslash\mbb{H}^2$ .

The first objective of this paper is to make explicit the connection (implicit in \cite{Hen} for $\mbb{Q}(\sqrt{-1})$ and explicit in \cite{Da2} for $\mbb{Q}(\sqrt{-3})$) ``badly approximable $\Longleftrightarrow$ bounded partial quotients'' for nearest integer continued fractions over $K$, where $K$ is any of the Euclidean imaginary quadratic fields, and to explore a class of complex numbers with ``atypical''  behavior, namely those lying on $K$-rational circles or lines, which include examples of algebraic numbers with bounded partial quotients (extending the results of \cite{BG}).  In particular, we prove the following.
\begin{itemize}
\item{\textbf{(Theorem \ref{thm:bndpartquots})}}
A number $z\in\mbb{C}$ is badly approximable over $K$ if and only if its partial quotients are bounded in norm.  Moreover an explicit approximation constant is given as a function of the bound on the partial quotients.
\item{\textbf{(Theorem \ref{thm:finite_orbit})}}
If $z\in\mbb{C}$ lies on a $K$-rational circle or line, (i.e. $(z,1)$ is a zero of the indefinite integral binary Hermitian form form $H(z,w)=Az\overline{z}-\overline{B}z\overline{w}-B\overline{z}w+Cw\overline{z}$, $A,C\in\mbb{Z}$, $B\in\mc{O}$), then its remainders $z_n=T^n(z)$ are ``atypical'' in that they lie on a finite number of lines and circular arcs (cf. Figures \ref{fig:circgauss}, \ref{fig:circeisen}).
\item{\textbf{(Corollary \ref{cor:aniso}, Corollary \ref{cor:bounds})}}
Moreover, if the rational circle on which $z$ lies does not contain any rational points, (i.e. the indefinite integral binary Hermitian form $H$ is anisotropic), then the remainders $z_n$ are bounded away from zero and the partial quotients $a_n$ are bounded in norm.  We give explicit bounds on $a_n$, $z_n$ in terms of $H$ and $K$.
\item{\textbf{(Corollary \ref{cor:algebraic})}}
There are algebraic numbers of every even degree over $\mbb{Q}$ that are badly approximable over $K$ (with effective approximation constant).  We also provide a characterization of these badly approximable algebraic numbers.
\end{itemize}

The second objective of this paper is to show that the main results above hold over every imaginary quadratic field, possibly non-Euclidean, but without effective constants.  Instead of using continued fractions, we employ a version of the Dani correspondence (Theorem \ref{thm:dani}) characterizing badly approximable numbers in terms of bounded geodesic trajectories in the Bianchi orbifolds.  In particular we have the following.
\begin{itemize}
\item{\textbf{(Theorem \ref{thm:badapproxcirc}, Corollary \ref{cor:algebraic})}}
Let $K$ be any imaginary quadratic field.  If $z\in\mbb{C}$ lies on a $K$-rational circle without rational points (i.e. $H(z,1)=0$ for an anisotropic indefinite binary Hermitian form with coefficients in $K$) then $z$ is badly approximable over $K$.  In particular, there are algebraic numbers of every even degree over $\mbb{Q}$ that are badly approximable over $K$, which we characterize.
\end{itemize}

In the mathoverflow post \cite{Di}, which inspired this work, the question is raised as to whether or not the examples of \cite{BG} exhaust the badly approximable algebraic numbers over $\mbb{Q}(\sqrt{-1})$.  Obvious ways to stay out of the cusps of $SL_2(\mc{O})\backslash\mbb{H}^3$ are to consider closed geodesics (anisotropic indefinite integral binary quadratic forms, i.e. quadratic irrationals) or compact geodesic surfaces (anisotropic integral indefinite binary Hermitian forms, giving the examples we explore in this paper).  Whether badly approximable numbers \textit{algebraic} over $K$ must be associated to compact geodesic surfaces in the Bianchi orbifolds is an interesting question (an extension of the folklore conjecture above), although it is not clear to the author why this should be so.
\subsection*{Acknowledgments}
The author would like to thank Katherine Stange for many helpful conversations and for reading various drafts of this document.


\section*{Nearest integer continued fractions over the Euclidean imaginary quadratic fields}
Let $K=K_d=\mbb{Q}(\sqrt{-d})$, $d>0$ a square-free integer, be an imaginary quadratic field and $\mc{O}=\mc{O}_d$ the ring of integers of $K$.  For $d=1,2,3,7,11$ the $\mc{O}_d$ are Euclidean with respect to the usual norm $|z|^2=z\overline{z}$, noting that the collection of disks $\{z\in\mbb{C} : |z-r|<1\}_{r\in\mc{O}}$ cover the plane, and in fact are the only $d$ for which $\mc{O}_d$ is Euclidean with respect to any function (cf. \cite{Le} \S4).  Consider the open Vorono\"{i} cell for $\mc{O}_d\subseteq\mbb{C}$, the collection of points closer to zero than to any other lattice point, along with a subset $\mc{E}$ of the boundary, so that we obtain a strict fundamental domain for the additive action of $\mc{O}$ on $\mbb{C}$,
$$
V=V_d=\{z\in\mbb{C} : |z|<|z-r|, \ r\in\mc{O}\}\cup\mc{E}, \ \mc{E}\subseteq\partial V.
$$
For the Euclidean values of $d$, and only for these values, $V_d$ is contained in the open unit disk.  The regions $V_d$ are rectangles for $d=1,2$ and hexagons for $d=3,7,11$; see Figure \ref{fig:basediag}.  For $z\in\mbb{C}$, we denote by $\floor{z}\in\mc{O}$ and $\rem{z}\in V$ the nearest integer and remainder, uniquely satisfying
$$
z=\floor{z}+\rem{z}.
$$
We now restrict ourselves to Euclidean $K$ to describe the continued fraction algorithm and applications, but we will return to arbitray imaginary quadratic $K$ in a later section.

We have an almost everywhere defined map $T=T_d:V_d\to V_d$ given by $T(z)=\rem{1/z}$.  For $z\in\mbb{C}$ define sequences $a_n\in\mc{O}$, $z_n\in V$, for $n\geq0$:
$$
a_0=\floor{z}, z_0=z-a_0=\rem{z}, \ a_n=\floor{\frac{1}{z_{n-1}}}, \ z_{n}=\rem{\frac{1}{z_{n-1}}}=\frac{1}{z_{n-1}}-a_n=T^n(z_0).
$$
In this way, we obtain a continued fraction expansion for $z\in\mbb{C}$,
$$
z=a_0+\frac{1}{a_1+\frac{1}{a_2+\dots}}=:[a_0;a_1,a_2,\ldots],
$$
where the expansion is finite for $z\in K$.  The convergents to $z$ will be denoted by
$$
\frac{p_n}{q_n}=[a_0;a_1,\ldots,a_n],
$$
where $p_n$, $q_n$ are defined by
$$
\left(\begin{array}{cc}
p_n&p_{n-1}\\q_n&q_{n-1}\\
\end{array}\right)
=\left(\begin{array}{cc}
a_0&1\\1&0\\
\end{array}\right)
\cdots
\left(
\begin{array}{cc}
a_n&1\\
1&0\\
\end{array}
\right).
$$
Here are a few easily verified algebraic properties that will be used below:
\begin{align*}
q_nz-p_n&=(-1)^nz_0\cdot\ldots\cdot z_n, \ z=\frac{p_n+z_np_{n-1}}{q_n+z_nq_{n-1}},\\
z-\frac{p_n}{q_n}&=\frac{(-1)^n}{q_n^2(z_n^{-1}+q_{n-1}/q_n)}, \ \frac{q_n}{q_{n-1}}=a_n+\frac{q_{n-2}}{q_{n-1}}.
\end{align*}
The first equality proves convergence $p_n/q_n\to z$ for irrational $z$ and gives a rate of convergence exponential in $n$.  A useful parameter is $\rho=\rho_d$, the radius of the smallest circle around zero containing $V_d$,
$$
\rho_d=\frac{\sqrt{1+d}}{2}, \ d=1,2, \ \rho_d=\frac{1+d}{4\sqrt{d}}, \ d=3,7,11.
$$
We note that $|a_n|\geq1/\rho_d$ for $n\geq1$, which is easily verified for each $d$.

Taking the transpose of the matrix expression above, we have the equality
$$
\frac{q_n}{q_{n-1}}=a_n+\frac{1}{a_{n-1}+\frac{1}{\dots+\frac{1}{a_1}}}\overset{alg.}{=}[a_n;a_{n-1},\ldots,a_1]
$$
as rational numbers (indicated by the overset ``alg.''), but this does not hold at the level of continued fractions, i.e. the continued fraction expansion of $q_{n-1}/q_n$ is not necessarily $[a_n;a_{n-1},\ldots,a_1]$.  See Figure \ref{fig:qratio} for the distribution of $q_{n-1}/q_n$, for 5000 random numbers and $1\leq n\leq10$, over $\mbb{Q}(\sqrt{-1})$ and $\mbb{Q}(\sqrt{-3})$.  The bounds $|q_{n+2}/q_n|\geq3/2$ are proved in \cite{Hen} and \cite{Da1} for $d=1$ and $3$ respectively.

\begin{figure}[hbt]
\begin{center}
\begin{tikzpicture}[scale=1]
\draw[color=blue](-2.5,2.5)--(2.5,2.5);
\draw[color=blue](-2.5,1.5)--(2.5,1.5);
\draw[color=blue](-2.5,0.5)--(2.5,0.5);
\draw[color=blue](-2.5,-0.5)--(2.5,-0.5);
\draw[color=blue](-2.5,-1.5)--(2.5,-1.5);
\draw[color=blue](-2.5,-2.5)--(2.5,-2.5);

\draw[color=blue](-2.5,-2.5)--(-2.5,2.5);
\draw[color=blue](-1.5,-2.5)--(-1.5,2.5);
\draw[color=blue](-0.5,-2.5)--(-0.5,2.5);
\draw[color=blue](0.5,-2.5)--(0.5,2.5);
\draw[color=blue](1.5,-2.5)--(1.5,2.5);
\draw[color=blue](2.5,-2.5)--(2.5,2.5);
\draw(0,0)circle(1);
\draw[color=red](1,1)arc(0:180:1);
\draw[color=red](-1,1)arc(90:270:1);
\draw[color=red](-1,-1)arc(180:360:1);
\draw[color=red](1,-1)arc(-90:90:1);

%
%
%
%
\end{tikzpicture}
\begin{tikzpicture}[scale=1]
\draw[color=blue](-2.5,3*1.414/2)--(-2.5,-3*1.414/2);
\draw[color=blue](-1.5,3*1.414/2)--(-1.5,-3*1.414/2);
\draw[color=blue](-.5,3*1.414/2)--(-.5,-3*1.414/2);
\draw[color=blue](.5,3*1.414/2)--(.5,-3*1.414/2);
\draw[color=blue](1.5,3*1.414/2)--(1.5,-3*1.414/2);
\draw[color=blue](2.5,3*1.414/2)--(2.5,-3*1.414/2);

\draw[color=blue](-2.5,3*1.414/2)--(2.5,3*1.414/2);
\draw[color=blue](-2.5,1.414/2)--(2.5,1.414/2);
\draw[color=blue](-2.5,-1.414/2)--(2.5,-1.414/2);
\draw[color=blue](-2.5,-3*1.414/2)--(2.5,-3*1.414/2);

\draw(0,0)circle(1);

\draw [red,domain=70.53-180:180-70.53,samples=500] plot ({cos(\x)+1}, {sin(\x)});
\draw [red,domain=70.53:360-70.53,samples=500] plot ({cos(\x)-1}, {sin(\x)});

\draw [red,domain=90-70.52:90+70.52,samples=500] plot ({1.414*cos(\x)/2}, {1.414*sin(\x)/2+1.414/2});
\draw [red,domain=-90-70.52:-90+70.52,samples=500] plot ({1.414*cos(\x)/2}, {1.414*sin(\x)/2-1.414/2});

\end{tikzpicture}
\begin{tikzpicture}[scale=1]
\draw[color=blue](0,0.577)--(1/2,0.288)--(1/2,-0.288)--(0,-0.577)--(-1/2,-0.288)--(-1/2,0.288)--(0,0.577);
\draw[color=blue](0+1,0.577)--(1/2+1,0.288)--(1/2+1,-0.288)--(0+1,-0.577)--(-1/2+1,-0.288)--(-1/2+1,0.288)--(0+1,0.577);
\draw[color=blue](0+2,0.577)--(1/2+2,0.288)--(1/2+2,-0.288)--(0+2,-0.577)--(-1/2+2,-0.288)--(-1/2+2,0.288)--(0+2,0.577);
\draw[color=blue](0-1,0.577)--(1/2-1,0.288)--(1/2-1,-0.288)--(0-1,-0.577)--(-1/2-1,-0.288)--(-1/2-1,0.288)--(0-1,0.577);
\draw[color=blue](0-2,0.577)--(1/2-2,0.288)--(1/2-2,-0.288)--(0-2,-0.577)--(-1/2-2,-0.288)--(-1/2-2,0.288)--(0-2,0.577);

\draw[color=blue](0-1/2,0.577+0.866)--(1/2-1/2,0.288+0.866)--(1/2-1/2,-0.288+0.866)--(0-1/2,-0.577+0.866)--(-1/2-1/2,-0.288+0.866)--(-1/2-1/2,0.288+0.866)--(0-1/2,0.577+0.866);
\draw[color=blue](0+1-1/2,0.577+0.866)--(1/2+1-1/2,0.288+0.866)--(1/2+1-1/2,-0.288+0.866)--(0+1-1/2,-0.577+0.866)--(-1/2+1-1/2,-0.288+0.866)--(-1/2+1-1/2,0.288+0.866)--(0+1-1/2,0.577+0.866);
\draw[color=blue](0+2-1/2,0.577+0.866)--(1/2+2-1/2,0.288+0.866)--(1/2+2-1/2,-0.288+0.866)--(0+2-1/2,-0.577+0.866)--(-1/2+2-1/2,-0.288+0.866)--(-1/2+2-1/2,0.288+0.866)--(0+2-1/2,0.577+0.866);
\draw[color=blue](0-1-1/2,0.577+0.866)--(1/2-1-1/2,0.288+0.866)--(1/2-1-1/2,-0.288+0.866)--(0-1-1/2,-0.577+0.866)--(-1/2-1-1/2,-0.288+0.866)--(-1/2-1-1/2,0.288+0.866)--(0-1-1/2,0.577+0.866);

\draw[color=blue](0-1/2,0.577-0.866)--(1/2-1/2,0.288-0.866)--(1/2-1/2,-0.288-0.866)--(0-1/2,-0.577-0.866)--(-1/2-1/2,-0.288-0.866)--(-1/2-1/2,0.288-0.866)--(0-1/2,0.577-0.866);
\draw[color=blue](0+1-1/2,0.577-0.866)--(1/2+1-1/2,0.288-0.866)--(1/2+1-1/2,-0.288-0.866)--(0+1-1/2,-0.577-0.866)--(-1/2+1-1/2,-0.288-0.866)--(-1/2+1-1/2,0.288-0.866)--(0+1-1/2,0.577-0.866);
\draw[color=blue](0+2-1/2,0.577-0.866)--(1/2+2-1/2,0.288-0.866)--(1/2+2-1/2,-0.288-0.866)--(0+2-1/2,-0.577-0.866)--(-1/2+2-1/2,-0.288-0.866)--(-1/2+2-1/2,0.288-0.866)--(0+2-1/2,0.577-0.866);
\draw[color=blue](0-1-1/2,0.577-0.866)--(1/2-1-1/2,0.288-0.866)--(1/2-1-1/2,-0.288-0.866)--(0-1-1/2,-0.577-0.866)--(-1/2-1-1/2,-0.288-0.866)--(-1/2-1-1/2,0.288-0.866)--(0-1-1/2,0.577-0.866);

\draw[color=blue](0,0.577+1.732)--(1/2,0.288+1.732)--(1/2,-0.288+1.732)--(0,-0.577+1.732)--(-1/2,-0.288+1.732)--(-1/2,0.288+1.732)--(0,0.577+1.732);
\draw[color=blue](0+1,0.577+1.732)--(1/2+1,0.288+1.732)--(1/2+1,-0.288+1.732)--(0+1,-0.577+1.732)--(-1/2+1,-0.288+1.732)--(-1/2+1,0.288+1.732)--(0+1,0.577+1.732);
\draw[color=blue](0-1,0.577+1.732)--(1/2-1,0.288+1.732)--(1/2-1,-0.288+1.732)--(0-1,-0.577+1.732)--(-1/2-1,-0.288+1.732)--(-1/2-1,0.288+1.732)--(0-1,0.577+1.732);

\draw[color=blue](0,0.577-1.732)--(1/2,0.288-1.732)--(1/2,-0.288-1.732)--(0,-0.577-1.732)--(-1/2,-0.288-1.732)--(-1/2,0.288-1.732)--(0,0.577-1.732);
\draw[color=blue](0+1,0.577-1.732)--(1/2+1,0.288-1.732)--(1/2+1,-0.288-1.732)--(0+1,-0.577-1.732)--(-1/2+1,-0.288-1.732)--(-1/2+1,0.288-1.732)--(0+1,0.577-1.732);
\draw[color=blue](0-1,0.577-1.732)--(1/2-1,0.288-1.732)--(1/2-1,-0.288-1.732)--(0-1,-0.577-1.732)--(-1/2-1,-0.288-1.732)--(-1/2-1,0.288-1.732)--(0-1,0.577-1.732);

\draw(0,0)circle(1);

\draw [red,domain=0:120,samples=500] plot ({cos(\x)+1/2}, {sin(\x)+0.866});
\draw [red,domain=60:180,samples=500] plot ({cos(\x)-1/2}, {sin(\x)+0.866});

\draw [red,domain=-120:0,samples=500] plot ({cos(\x)+1/2}, {sin(\x)-0.866});
\draw [red,domain=-180:-60,samples=500] plot ({cos(\x)-1/2}, {sin(\x)-0.866});

\draw [red,domain=-60:60,samples=500] plot ({cos(\x)+1}, {sin(\x)});
\draw [red,domain=120:240,samples=500] plot ({cos(\x)-1}, {sin(\x)});

\end{tikzpicture}
\end{center}
\begin{center}
\begin{tikzpicture}[scale=1]
\draw[color=blue](0,2/2.645)--(1/2,3/2/2.645)--(1/2,-3/2/2.645)--(0,-2/2.645)--(-1/2,-3/2/2.645)--(-1/2,3/2/2.645)--(0,2/2.645);
\draw[color=blue](0+1,2/2.645)--(1/2+1,3/2/2.645)--(1/2+1,-3/2/2.645)--(0+1,-2/2.645)--(-1/2+1,-3/2/2.645)--(-1/2+1,3/2/2.645)--(0+1,2/2.645);
\draw[color=blue](0-1,2/2.645)--(1/2-1,3/2/2.645)--(1/2-1,-3/2/2.645)--(0-1,-2/2.645)--(-1/2-1,-3/2/2.645)--(-1/2-1,3/2/2.645)--(0-1,2/2.645);
\draw[color=blue](0+2,2/2.645)--(1/2+2,3/2/2.645)--(1/2+2,-3/2/2.645)--(0+2,-2/2.645)--(-1/2+2,-3/2/2.645)--(-1/2+2,3/2/2.645)--(0+2,2/2.645);
\draw[color=blue](0-2,2/2.645)--(1/2-2,3/2/2.645)--(1/2-2,-3/2/2.645)--(0-2,-2/2.645)--(-1/2-2,-3/2/2.645)--(-1/2-2,3/2/2.645)--(0-2,2/2.645);

\draw[color=blue](0-1/2,2/2.645+1.323)--(1/2-1/2,3/2/2.645+1.323)--(1/2-1/2,-3/2/2.645+1.323)--(0-1/2,-2/2.645+1.323)--(-1/2-1/2,-3/2/2.645+1.323)--(-1/2-1/2,3/2/2.645+1.323)--(0-1/2,2/2.645+1.323);
\draw[color=blue](0+1-1/2,2/2.645+1.323)--(1/2+1-1/2,3/2/2.645+1.323)--(1/2+1-1/2,-3/2/2.645+1.323)--(0+1-1/2,-2/2.645+1.323)--(-1/2+1-1/2,-3/2/2.645+1.323)--(-1/2+1-1/2,3/2/2.645+1.323)--(0+1-1/2,2/2.645+1.323);
\draw[color=blue](0-1-1/2,2/2.645+1.323)--(1/2-1-1/2,3/2/2.645+1.323)--(1/2-1-1/2,-3/2/2.645+1.323)--(0-1-1/2,-2/2.645+1.323)--(-1/2-1-1/2,-3/2/2.645+1.323)--(-1/2-1-1/2,3/2/2.645+1.323)--(0-1-1/2,2/2.645+1.323);
\draw[color=blue](0+2-1/2,2/2.645+1.323)--(1/2+2-1/2,3/2/2.645+1.323)--(1/2+2-1/2,-3/2/2.645+1.323)--(0+2-1/2,-2/2.645+1.323)--(-1/2+2-1/2,-3/2/2.645+1.323)--(-1/2+2-1/2,3/2/2.645+1.323)--(0+2-1/2,2/2.645+1.323);

\draw[color=blue](0-1/2,2/2.645-1.323)--(1/2-1/2,3/2/2.645-1.323)--(1/2-1/2,-3/2/2.645-1.323)--(0-1/2,-2/2.645-1.323)--(-1/2-1/2,-3/2/2.645-1.323)--(-1/2-1/2,3/2/2.645-1.323)--(0-1/2,2/2.645-1.323);
\draw[color=blue](0+1-1/2,2/2.645-1.323)--(1/2+1-1/2,3/2/2.645-1.323)--(1/2+1-1/2,-3/2/2.645-1.323)--(0+1-1/2,-2/2.645-1.323)--(-1/2+1-1/2,-3/2/2.645-1.323)--(-1/2+1-1/2,3/2/2.645-1.323)--(0+1-1/2,2/2.645-1.323);
\draw[color=blue](0-1-1/2,2/2.645-1.323)--(1/2-1-1/2,3/2/2.645-1.323)--(1/2-1-1/2,-3/2/2.645-1.323)--(0-1-1/2,-2/2.645-1.323)--(-1/2-1-1/2,-3/2/2.645-1.323)--(-1/2-1-1/2,3/2/2.645-1.323)--(0-1-1/2,2/2.645-1.323);
\draw[color=blue](0+2-1/2,2/2.645-1.323)--(1/2+2-1/2,3/2/2.645-1.323)--(1/2+2-1/2,-3/2/2.645-1.323)--(0+2-1/2,-2/2.645-1.323)--(-1/2+2-1/2,-3/2/2.645-1.323)--(-1/2+2-1/2,3/2/2.645-1.323)--(0+2-1/2,2/2.645-1.323);

\draw(0,0)circle(1);

\draw[color=red](7/8,-0.992)arc(-97.18:97.18:1);
\draw[color=red](-7/8,0.992)arc(180-97.18:180+97.18:1);
\draw[color=red](7/8,0.992)arc(27.89:27.89+82.82:0.707);
\draw[color=red](0,1.323)arc(180-27.89-82.82:180-27.89:0.707);
\draw[color=red](-7/8,-0.992)arc(180+27.89:180+27.89+82.82:0.707);
\draw[color=red](0,-1.323)arc(-27.89-82.82:-27.89:0.707);

\end{tikzpicture}
\begin{tikzpicture}[scale=1]
\draw[color=blue](0,3/3.316)--(1/2,5/2/3.316)--(1/2,-5/2/3.316)--(0,-3/3.316)--(-1/2,-5/2/3.316)--(-1/2,5/2/3.316)--(0,3/3.316);
\draw[color=blue](0+1,3/3.316)--(1/2+1,5/2/3.316)--(1/2+1,-5/2/3.316)--(0+1,-3/3.316)--(-1/2+1,-5/2/3.316)--(-1/2+1,5/2/3.316)--(0+1,3/3.316);
\draw[color=blue](0-1,3/3.316)--(1/2-1,5/2/3.316)--(1/2-1,-5/2/3.316)--(0-1,-3/3.316)--(-1/2-1,-5/2/3.316)--(-1/2-1,5/2/3.316)--(0-1,3/3.316);
\draw[color=blue](0+2,3/3.316)--(1/2+2,5/2/3.316)--(1/2+2,-5/2/3.316)--(0+2,-3/3.316)--(-1/2+2,-5/2/3.316)--(-1/2+2,5/2/3.316)--(0+2,3/3.316);
\draw[color=blue](0-2,3/3.316)--(1/2-2,5/2/3.316)--(1/2-2,-5/2/3.316)--(0-2,-3/3.316)--(-1/2-2,-5/2/3.316)--(-1/2-2,5/2/3.316)--(0-2,3/3.316);

\draw[color=blue](0-1/2,3/3.316+1.658)--(1/2-1/2,5/2/3.316+1.658)--(1/2-1/2,-5/2/3.316+1.658)--(0-1/2,-3/3.316+1.658)--(-1/2-1/2,-5/2/3.316+1.658)--(-1/2-1/2,5/2/3.316+1.658)--(0-1/2,3/3.316+1.658);
\draw[color=blue](0+1-1/2,3/3.316+1.658)--(1/2+1-1/2,5/2/3.316+1.658)--(1/2+1-1/2,-5/2/3.316+1.658)--(0+1-1/2,-3/3.316+1.658)--(-1/2+1-1/2,-5/2/3.316+1.658)--(-1/2+1-1/2,5/2/3.316+1.658)--(0+1-1/2,3/3.316+1.658);
\draw[color=blue](0-1-1/2,3/3.316+1.658)--(1/2-1-1/2,5/2/3.316+1.658)--(1/2-1-1/2,-5/2/3.316+1.658)--(0-1-1/2,-3/3.316+1.658)--(-1/2-1-1/2,-5/2/3.316+1.658)--(-1/2-1-1/2,5/2/3.316+1.658)--(0-1-1/2,3/3.316+1.658);
\draw[color=blue](0+2-1/2,3/3.316+1.658)--(1/2+2-1/2,5/2/3.316+1.658)--(1/2+2-1/2,-5/2/3.316+1.658)--(0+2-1/2,-3/3.316+1.658)--(-1/2+2-1/2,-5/2/3.316+1.658)--(-1/2+2-1/2,5/2/3.316+1.658)--(0+2-1/2,3/3.316+1.658);

\draw[color=blue](0-1/2,3/3.316-1.658)--(1/2-1/2,5/2/3.316-1.658)--(1/2-1/2,-5/2/3.316-1.658)--(0-1/2,-3/3.316-1.658)--(-1/2-1/2,-5/2/3.316-1.658)--(-1/2-1/2,5/2/3.316-1.658)--(0-1/2,3/3.316-1.658);
\draw[color=blue](0+1-1/2,3/3.316-1.658)--(1/2+1-1/2,5/2/3.316-1.658)--(1/2+1-1/2,-5/2/3.316-1.658)--(0+1-1/2,-3/3.316-1.658)--(-1/2+1-1/2,-5/2/3.316-1.658)--(-1/2+1-1/2,5/2/3.316-1.658)--(0+1-1/2,3/3.316-1.658);
\draw[color=blue](0-1-1/2,3/3.316-1.658)--(1/2-1-1/2,5/2/3.316-1.658)--(1/2-1-1/2,-5/2/3.316-1.658)--(0-1-1/2,-3/3.316-1.658)--(-1/2-1-1/2,-5/2/3.316-1.658)--(-1/2-1-1/2,5/2/3.316-1.658)--(0-1-1/2,3/3.316-1.658);
\draw[color=blue](0+2-1/2,3/3.316-1.658)--(1/2+2-1/2,5/2/3.316-1.658)--(1/2+2-1/2,-5/2/3.316-1.658)--(0+2-1/2,-3/3.316-1.658)--(-1/2+2-1/2,-5/2/3.316-1.658)--(-1/2+2-1/2,5/2/3.316-1.658)--(0+2-1/2,3/3.316-1.658);

\draw(0,0)circle(1);

\draw[color=red](11/18,-0.9213)arc(-112.9:112.9:1);
\draw[color=red](-11/18,0.9213)arc(180-112.9:180+112.9:1);
\draw[color=red](11/18,0.9213)arc(39.66:39.66+67.12:0.577);
\draw[color=red](-11/18,0.9213)arc(180-39.66:180-39.66-67.12:0.577);

\draw[color=red](11/18,-0.9213)arc(-39.66:-39.66-67.12:0.577);
\draw[color=red](-11/18,-0.9213)arc(180+39.66:180+39.66+67.12:0.577);

\end{tikzpicture}
\caption{$\partial V$ and translates (blue), $\partial(V^{-1})$ (red), and unit circle (black) for $d=1,2,3,7,11$.}
\label{fig:basediag}
\end{center}
\end{figure}
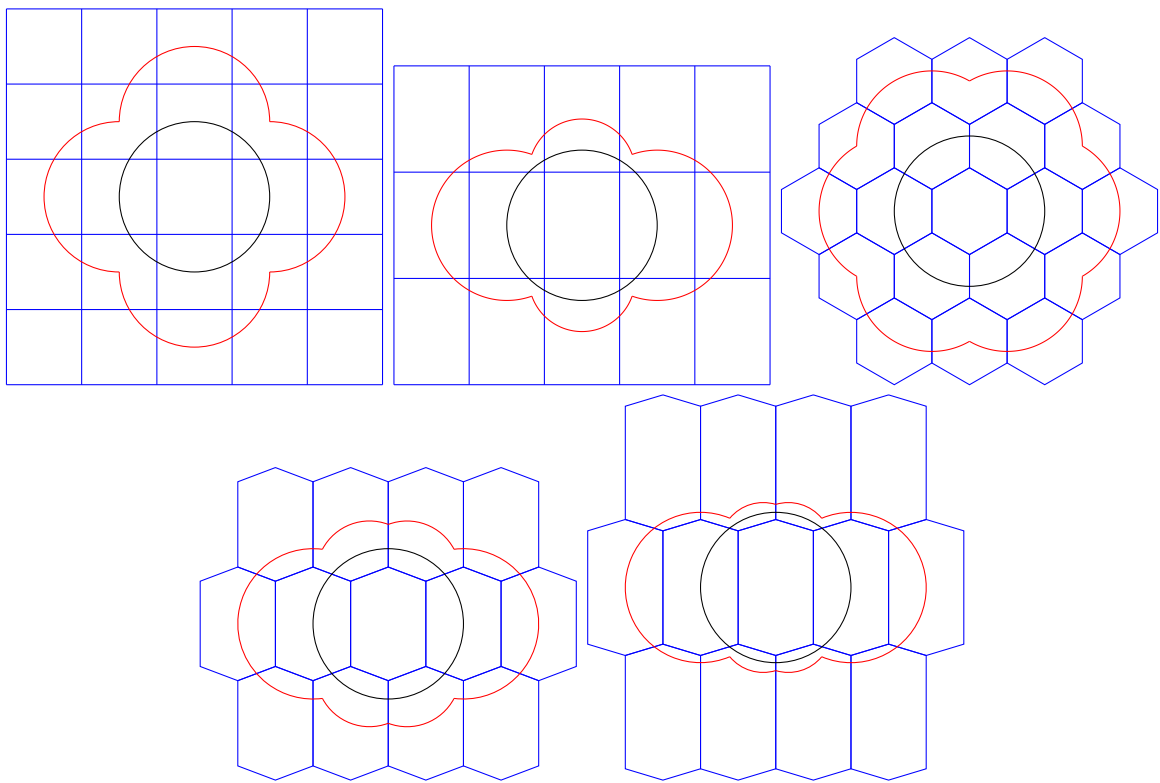

\begin{figure}[hbt]
\begin{center}
\includegraphics[scale=0.5]{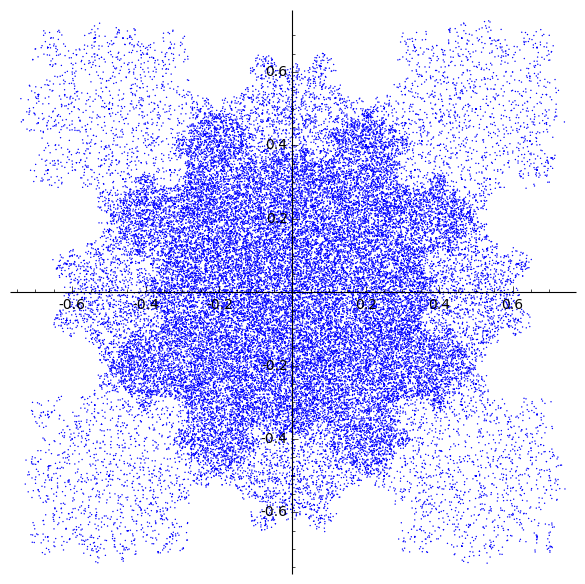}
\includegraphics[scale=0.5]{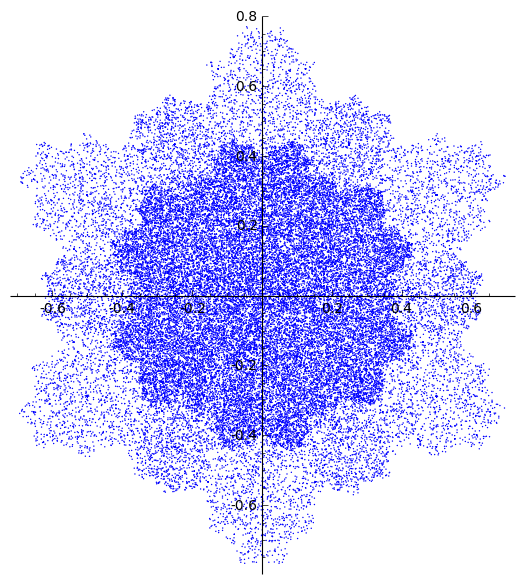}
\caption{The numbers $q_{n-1}/q_n$, $1\leq n\leq10$, for 5000 randomly chosen $z$ over $\mbb{Q}(\sqrt{-1})$ and $\mbb{Q}(\sqrt{-3})$.}
\label{fig:qratio}
\end{center}
\end{figure}

Monotonicity of the denominators $q_n$ was shown by Hurwitz \cite{Hur1} for $d=1,3$, Lunz \cite{Lu} for $d=2$, and stated without proof in \cite{La} for $d=1,2,3,7,11$.  As this is a desirable property to establish, we outline the proof for the cases $d=7,11$ in an appendix.  The proofs are unenlightening and follow the outline for the simpler cases $d=1,3$ in \cite{Hur1}.

\begin{prop}\label{prop:mono}
For any $z\in\mbb{C}$, the denominators of the convergents $p_n/q_n$ are strictly increasing in absolute value, $|q_{n-1}|<|q_n|$.
\end{prop}
\begin{proof}
See the appendix.
\end{proof}

To conclude this section, we record the following lemma, which is used in the proof of Theorem \ref{thm:bndpartquots}, applied to the inverse of $g_n=\left(
\begin{array}{cc}
p_n&p_{n-1}\\
q_n&q_{n-1}\\
\end{array}
\right)$,
for which $g_n(\infty)=p_n/q_n$ and $g_n^{-1}(\infty)=-q_{n-1}/q_n$ as depicted in Figure \ref{fig:diag}.
\begin{lemma}\label{lemma:normalize}
Let $w=g(z)=\frac{az+b}{cz+d}$ with $a,b,c,d\in\mc{O}$, $|ad-bc|=1$, and $g(p/q)=\infty$ (i.e. $p/q=-d/c$).  Then the disk $D=\{z\in\mbb{C} : |z-p/q|<C/|q|^2\}$ gets mapped via $g$ to the region $g(D)=\{w\in\mbb{C} : |w-a/c|>1/C\}$, the exterior of the disk of radius $1/C$ centered at $g(\infty)$.
\end{lemma}
\begin{proof}
We have
$$
w-a/c=\frac{az+b}{cz+d}-\frac{a}{c}=\frac{-\det(g)}{c^2(z+d/c)}, \ |w-a/c|=\frac{1}{|c|^2|z+d/c|}=\frac{1}{|q|^2|z-p/q|},
$$
so that
$$
|w-a/c|>1/C\Longleftrightarrow\frac{1}{|q|^2|z-p/q|}>1/C\Longleftrightarrow |z-p/q|<C/|q|^2.
$$
\end{proof}

Some references for nearest integer continued fractions include: \cite{Hur1} (some generalities and $d=1,3$), \cite{P} $\S46$ ($d=1,2,3$), \cite{La} ($d=1,2,3,7,11$), \cite{Hen} Chapter 5 ($d=1$), and \cite{Da1} (a general approach including some properties of the cases we consider).
\section*{Badly approximable numbers over the Euclidean imaginary quadratic fields via nearest integer continued fractions}
For each of the Euclidean imaginary quadratic fields $K$ there is a constant $C>0$ such that for any $z\in\mbb{C}$ there are infinitely many solutions $p/q\in K$, $(p,q)=1$ to
\begin{align}
\label{eqn:approx}
|z-p/q|\leq C/|q|^2,\tag{\dag}
\end{align}
by a pigeonhole argument for instance (cf. \cite{EGM} Chapter 7, Proposition 2.6).  The smallest such $C$ are $1/\sqrt{3}$, $1/\sqrt{2}$, $1/\sqrt[4]{13}$, $1/\sqrt[4]{8}$, and $2/\sqrt{5}$ for $d=1,2,3,7,11$ respectively (for references, see the Introduction to \cite{Vu1}).  We can obtain rational approximations with a specific $C$ satisfying inequality \eqref{eqn:approx} using the nearest integer algorithms described above.  The best constants coming from the nearest integer convergents, $\sup_{z,n}\{|q_n|^2|z-p_n/q_n|\}$, can be found in Theorem 1 of \cite{La}.
\begin{prop}\label{prop:approx_const}
For $z\in\mbb{C}\setminus K$, the convergents $p_n/q_n$ satisfy
$$
|z-p_n/q_n|\leq\frac{1}{(1/\rho-1)|q_n|^2},
$$
i.e. we can take $p/q=p_n/q_n$ and $C=\frac{\rho}{1-\rho}$ in the inequality \eqref{eqn:approx}.
\end{prop}
\begin{proof}
Using simple properties of the algorithm and the bounds $1/z_n\in V^{-1}$, $|q_{n-1}/q_n|\leq1$, we have
$$
|z-p_n/q_n|=\frac{1}{|q_n|^2|z_n^{-1}+q_{n-1}/q_n|}\leq\frac{1}{|q_n|^2(1/\rho-1)}.
$$
\end{proof}

We say $z$ is \textit{badly approximable} over $K$ if there is a $C'>0$ such that
$$
|z-p/q|\geq C'/|q|^2
$$
for all $p/q\in K$, i.e. $z$ is badly approximable if the exponent of two on $|q|$ is the best possible in the inequality \eqref{eqn:approx}.  We will show that the badly approximable numbers are characterized by the boundedness of the partial quotients in the nearest integer continued fraction expansion, analogous to the well-known fact for simple continued fractions over the real numbers.  First a lemma showing that the nearest integer convergents compare well with the best rational approximations.
\begin{lemma}\label{lemma:multconst}
There are effective constants $\alpha=\alpha_d>0$ such that for any irrational $z$ with convergents $p_n/q_n$ and rational $p/q$ with $|q_{n-1}|<|q|\leq|q_n|$ we have
$$
|q_nz-p_n|\leq\alpha|qz-p|.
$$
\end{lemma}
\begin{proof}
Write $p/q$ in terms of the convergents $p_n/q_n$, $p_{n-1}/q_{n-1}$ for some $s,t\in\mc{O}$
\begin{align*}
\left(
\begin{array}{c}
p\\
q\\
\end{array}
\right)
=
\left(
\begin{array}{cc}
p_n&p_{n-1}\\
q_n&q_{n-1}\\
\end{array}
\right)
\left(
\begin{array}{c}
s\\
t\\
\end{array}
\right)=
\left(
\begin{array}{c}
p_ns+p_{n-1}t\\
q_ns+q_{n-1}t\\
\end{array}
\right)
.
\end{align*}
If $s=0$, then $p/q=p_{n-1}/q_{n-1}$, impossible by the assumption $|q_{n-1}|<|q|$.  If $t=0$, then $p/q=p_n/q_n$ and the result is clear with $\alpha=1$.  We may therefore assume $|s|,|t|\geq1$.  We have
$$
\left|z-\frac{p}{q}\right|\geq\left|\left|\frac{p_n}{q_n}-\frac{p}{q}\right|-\left|z-\frac{p_n}{q_n}\right|\right|=
\left|\left|\frac{t}{qq_n}\right|-\left|z-\frac{p_n}{q_n}\right|\right|,
$$
noting that $t=(-1)^n(pq_n-p_nq)$ by inverting the matrix relating $p$, $q$, $s$, and $t$.

Define $\delta$ by $|t|=\delta|q_n|^2|z-p_n/q_n|$, so that
$$
\frac{1}{|\delta-|q/q_{n}||}|qz-p|\geq|q_nz-p_n|.
$$
If $\delta>1$, then we have our $\alpha$.  A lower bound for $\delta$ is
$$
\delta=\frac{|t|}{|q_n|^2|z-p_n/q_n|}\geq|t|\inf_{z,n}\{(|q_n||q_nz-p_n|)^{-1}\}.
$$
The infimum above is calculated in \cite{La}, Theorem 1, where it is found to be
$$
\inf_{z,n}\{(|q_n||q_nz-p_n|)^{-1}\}=\left\{
\begin{array}{cc}
1&d=1,\\
\sqrt{\frac{486-\sqrt{3}}{786}}=0.78493\ldots&d=2,\\
\sqrt{\frac{7+\sqrt{21}}{7}}=1.28633\ldots&d=3,\\
\sqrt{\frac{2093-9\sqrt{21}}{2408}}=0.92307\ldots&d=7,\\
\frac{1}{5\sqrt{2}}\sqrt{30-8\sqrt{5}-5\sqrt{11}+3\sqrt{55}}=0.59627\ldots&d=11.
\end{array}
\right.
$$
The smallest integers of norm greater than one in $\mc{O}_d$ have absolute values of $\sqrt{2}$ (for $d=1,2,7$) and $\sqrt{3}$ (for $d=3,11$).  Multiplying these potential values of $|t|$ by the above constants gives values of $\delta$ greater than one, so that $|\delta-|q/q_n||\geq|\delta-1|$ is bounded away from zero.  Hence we are left to explore those rationals $p/q$ with $|t|=1$.

For general $t$ we have
$$
qz-p=q\frac{p_n+z_np_{n-1}}{q_n+z_nq_{n-1}}-p=(q_ns+tq_{n-1})\frac{p_n+z_np_{n-1}}{q_n+z_nq_{n-1}}-(p_ns+tp_{n-1})=\frac{(-1)^n(sz_n-t)}{q_n+z_nq_{n-1}}
$$
and
$$
q_nz-p_n=\frac{(-1)^nz_n}{q_n+z_nq_{n-1}},
$$
and we want $\alpha>0$ such that
$$
|q_nz-p_n|\leq\alpha|qz-p|.
$$
Substituting the above we have
$$
|q_nz-p_n|\leq\alpha|qz-p|\Longleftrightarrow\frac{|z_n|}{|q_n+z_nq_{n-1}|}\leq\alpha\frac{|z_ns-t|}{|q_n+z_nq_{n-1}|}\Longleftrightarrow\frac{1}{|s-t/z_n|}\leq\alpha.
$$
If $|s-t/z_n|<1/2$ and $|t|=1$, then $s/t=a_{n+1}$ since $s/t\in\mc{O}$ is the nearest integer to $1/z_n$.  However (with $|q_{n-1}|<|q|\leq|q_{n}|$, $q=sq_{n}+tq_{n-1}$),
$$
\left|\frac{q_{n+1}}{q_{n}}\right|=\left|a_{n+1}+\frac{q_{n-1}}{q_{n}}\right|=\left|\frac{s}{t}+\frac{q_{n-1}}{q_{n}}\right|=\left|s+t\frac{q_{n-1}}{q_{n}}\right|=\left|\frac{q}{q_{n}}\right|\leq1,
$$
and we obtain a contradiction if $|s-t/z_n|<1/2$ and $|t|=1$.  Hence when $|t|=1$ we can take $\alpha=2$.

In summary, we can take
$$
\alpha_d=\left\{
\begin{array}{cc}
2.41421\ldots&d=1\\
9.08592\ldots&d=2\\
2&d=3\\
3.27419\ldots&d=7\\
30.51490\ldots&d=11\\
\end{array}
\right.,
$$
taking the maximum of $2$ (covering the case $|t|=1$) and the bound on $\frac{1}{\delta-|q/q_n|}$ for $|t|>1$.
\end{proof}
No attempt was made to optimize the value of $\alpha$ in the lemma.  The above result for $d=1,3$ and $\alpha=1$ is contained in Theorem 2 of \cite{La}.  Another proof for $d=1$ and $\alpha=5$ is Theorem 5.1 of \cite{Hen}, and a proof for $d=3$, $\alpha=2$ can be found in \cite{Da2}.  The purpose of the above lemma is to establish the following proposition (which for $d=3$ is Corollary 1.3 of \cite{Da2}).
\begin{theorem}\label{thm:bndpartquots}
A number $z\in\mbb{C}\setminus K$ is badly approximable if and only if its partial quotients $a_n$ are bounded (if and only if the remainders $z_n$ are bounded away from zero).  In particular, if $|a_n|\leq\beta$ for all $n$ and $p/q\in K$, then $|z-p/q|\geq C'/|q|^2$ where
$$
C'=\frac{1}{\alpha(\beta+1)(\beta+\rho+1)}.
$$
\end{theorem}
\begin{proof}
If $z$ is badly approximable, then there is a $C'>0$ such that for each convergent $p_n/q_n$ to $z$, the disk $|w-p_n/q_n|\leq C'/|q_n|^2$ does not contain $z$.  Mapping $p_n/q_n$ to $\infty$ via $g_n^{-1}$, where
$$
g_n=\left(
\begin{array}{cc}
a_0&1\\
1&0\\
\end{array}
\right)\ldots
\left(
\begin{array}{cc}
a_n&1\\
1&0\\
\end{array}
\right)=
\left(
\begin{array}{cc}
p_n&p_{n-1}\\
q_n&q_{n-1}\\
\end{array}
\right),
$$
maps the disk $|w-p_n/q_n|\leq C'/|q_n|^2$ to the region $|w+q_{n-1}/q_n|\geq1/C'$, centered at $g_n^{-1}(\infty)=-q_{n-1}/q_n$ (cf. Lemma \ref{lemma:normalize}).  Because $g_n^{-1}(z)$ is inside the disk of radius $1/C'$ centered at $-q_{n-1}/q_n$ and $|-q_{n-1}/q_n|<1$, we have
\begin{align*}
a_{n+1}+z_{n+1}&=1/z_n=g_n^{-1}(z),\\
|a_{n+1}|&\leq|z_{n+1}|+|g_n^{-1}(z)|\leq\rho+1+1/C'.
\end{align*}
Hence $a_{n+1}$ is bounded.  See Figure \ref{fig:diag} below for an illustration.

By Lemma \ref{lemma:multconst}, for $z$ and $p/q$ with $|q_{n-1}|<|q|\leq|q_n|$ we have
\begin{align*}
\left|z-\frac{p_n}{q_n}\right|&\leq\alpha\left|z-\frac{p}{q}\right|\left|\frac{q}{q_n}\right|\leq\alpha\left|z-\frac{p}{q}\right|\frac{|q|^2}{|q_n|^2}\left|\frac{q_n}{q_{n-1}}\right|\\
&=\alpha\left|z-\frac{p}{q}\right|\frac{|q|^2}{|q_n|^2}\left|a_n+\frac{q_{n-2}}{q_{n-1}}\right|\\
&\leq\alpha\left|z-\frac{p}{q}\right|\frac{|q|^2}{|q_n|^2}(|a_n|+1),\\
|q_n|^2\left|z-\frac{p_n}{q_n}\right|&\leq\alpha(|a_n|+1)|q|^2\left|z-\frac{p}{q}\right|.
\end{align*}
This shows that if $z$ has bounded partial quotients, then $z$ is badly approximable if and only if it is badly approximable \textit{by its convergents}.  For approximation by convergents, we have
\begin{align*}
\left|z-\frac{p_n}{q_n}\right|=\frac{1}{|q_n|^2|z_n^{-1}+q_{n-1}/q_n|}=\frac{1}{|q_n|^2|a_{n+1}+z_{n+1}+q_{n-1}/q_n|}\geq\frac{1}{|q_n|^2(|a_{n+1}|+\rho+1)},
\end{align*}
showing that if the partial quotients of $z$ are bounded, then $z$ is badly approximable by convergents and therefore badly approximable.  For an approximation constant, the above discussion gives $|z-p/q|\geq C'/|q|^2$ for any $p/q\in K$ where
$$
C'=\frac{1}{\alpha(\beta+1)(\beta+\rho+1)}
$$
and $\beta$ is an upper bound for the $a_n$.
\end{proof}

\begin{figure}[hbt]
\begin{center}
\includegraphics[scale=.5]{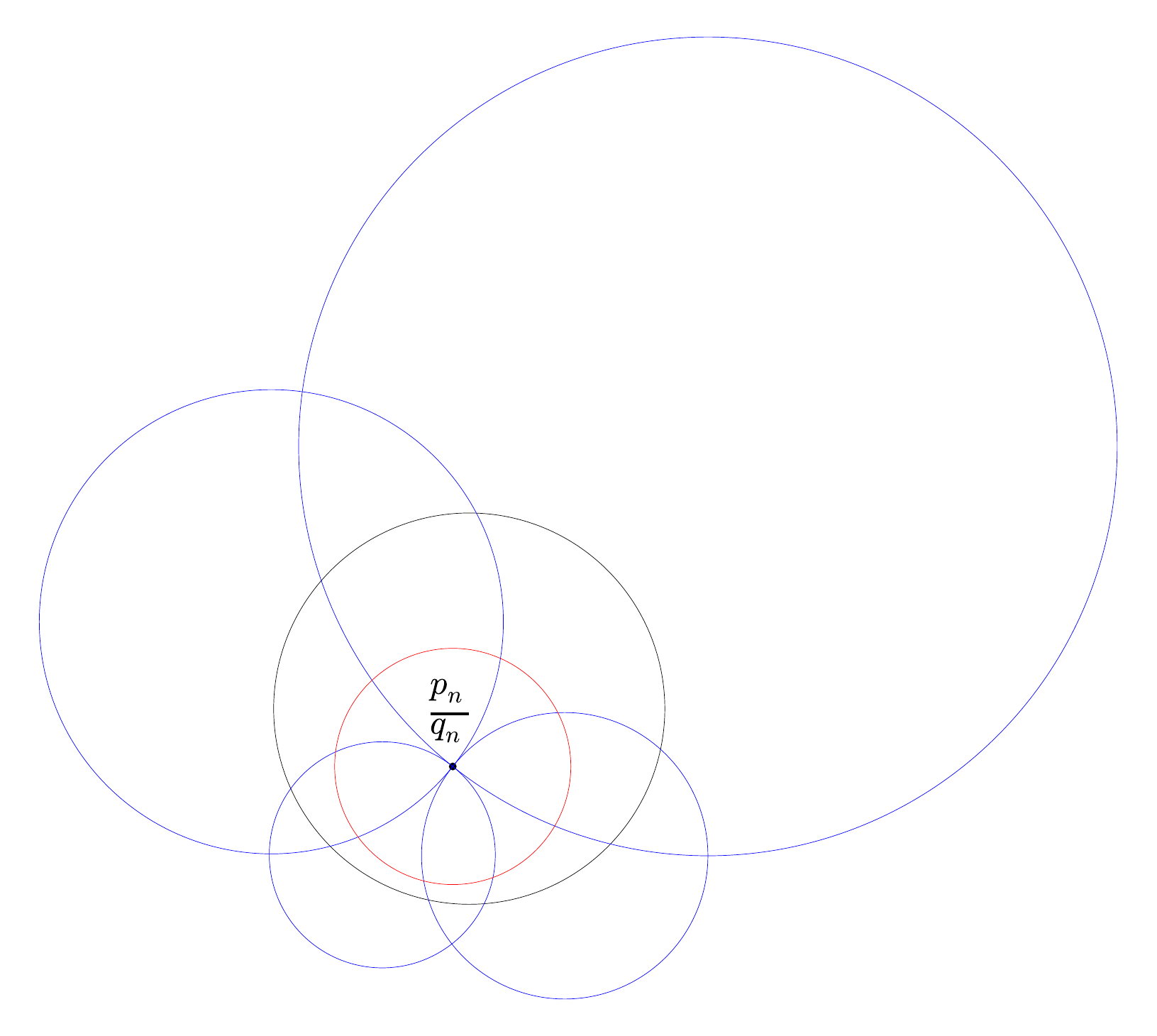}
\includegraphics[scale=.5]{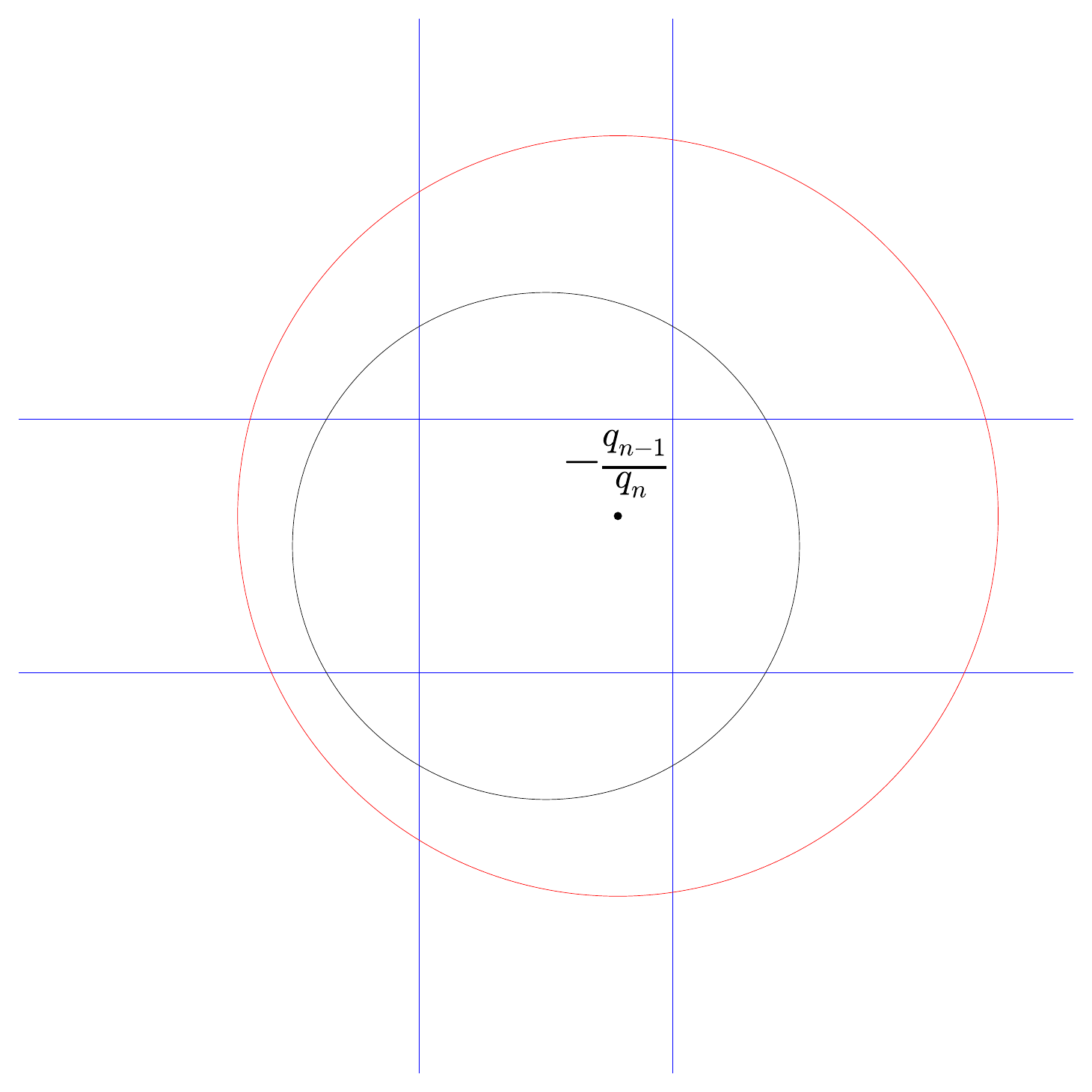}
\caption{Over $\mbb{Q}(\sqrt{-1})$, we have the points $p_n/q_n=g_n(\infty)$ and $-q_{n-1}/q_n=g_n^{-1}(\infty)$, along with the unit circle and its image under $g_n$ (black), circles of radius $1/C'$ and $C'/|q_n|^2$ (red), and the lines defining $V$ and their images under $g_n$ (blue).}
\label{fig:diag}
\end{center}
\end{figure}

\section*{The continued fraction expansions of points on $K$-rational circles}

In this section we focus on producing $z$ with bounded partial quotients extending the results of \cite{BG} to all of the Euclidean imaginary quadratic fields $K_d$.  We will show that there are many circles in the complex plane all of whose points have bounded partial quotients.

We will consider equivalence classes of indefinite integral binary Hermitian forms.  A \textit{binary Hermitian form} $H(z,w)$ is a function of the form
$$
H(z,w)=(\overline{z},\overline{w})\left(\begin{array}{cc}A&-B\\-\overline{B}&C\\\end{array}\right)\left(\begin{array}{c}z\\w\end{array}\right)=Az\overline{z}-B\overline{z}w-\overline{B}z\overline{w}+Cw\overline{w}, \ A,C\in\mbb{R}, \ B\in\mbb{C}.
$$
We denote by $\Delta(H)$ the determinant $\det(H)=AC-|B|^2$ of the Hermitian matrix defining $H$.  The binary Hermitian form $H$ is \textit{integral} over $K$ if the matrix entries of $H$ are integers, i.e. $A,C\in\mbb{Z}$ and $B\in\mc{O}$.  The form is \textit{indefinite} (takes on both positive and negative values) if and only if $\Delta(H)<0$.  The zero set of an indefinite $H$ on the Riemann sphere $P^1(\mbb{C})$ is a circle (using homogeneous coordinates $[z:w]$ on the projective line)
$$
Z(H):=\{[z:w]\in P^1(\mbb{C}) : H(z,w)=0\}
$$
which is either a circle or a line in the chart $\mbb{C}_z=\{[z:1]\in P^1(\mbb{C})\}$
$$
Z(H)\cap\mbb{C}_z=\left\{
\begin{array}{cc}
\{z : |z-B/A|^2=-\Delta/A^2\} &\ \text{ if } A\neq0,\\
\{z : \text{Re}(\overline{B}z)=C\} &\ \text{ if } A=0.\\
\end{array}\right.
$$
We will be interested in equivalence of forms over $GL_2(\mc{O})$, where $g\in GL_2(\mbb{C})$ acts as a normalized linear change of variable on the left, ${}^gH=|\det(g)|(g^{-1})^*Hg^{-1}$ (here $*$ denotes the conjugate transpose), and also with the M\"{o}bius action of $GL_2(\mbb{C})$ on $P^1(\mbb{C})$,
$$
g\cdot[z:w]=[az+bw:cz+dw], \ g=\left(\begin{array}{cc}a&b\\c&d\\\end{array}\right).
$$
We collect some easily verified facts in the following lemma.
\begin{lemma}
The following hold for the action ${}^gH=|\det(g)|(g^{-1})^*Hg^{-1}$, $g\in GL_2(\mbb{C})$, on indefinite binary Hermitian forms.
\begin{itemize}
\item
The action of $GL_2(\mbb{C})$ is determinant preserving, i.e. $\Delta({}^gH)=\Delta(H)$.
\item
The map $H\mapsto Z(H)$ is $GL_2(\mbb{C})$-equivariant (i.e. $g\cdot Z(H)=Z({}^gH)$).
\end{itemize}
Furthermore, an integral form $H$ is isotropic (i.e. $H(z,w)=0$ for some $[z:w]\in P^1(K)$) if and only if $-\Delta(H)$ is in the image of the norm map $N^K_{\mbb{Q}}:K\to\mbb{Q}$.
\end{lemma}
\begin{proof}
For $g\in SL_2(\mbb{C})$ we have
$$
\det({}^gH)=|\det(g)|^2\det((g^{-1})^*Hg^{-1})=\frac{|\det(g)|^2\det(H)}{\det(g)\overline{\det(g)}}=\det(H).
$$
The second bullet follows from
$$
((\bar{z}\ \bar{w})g^*)({}^gH)(g(z\ w)^t)=|\det(g)|((\bar{z}\ \bar{w})g^*)((g^{-1})^*Hg^{-1})g(z,w)^t=|\det(g)|H(z,w).
$$
Finally, the factorization (assuming $A\neq0$ else $-\Delta=|B|^2$ is a norm and $H(1,0)=H(0,1)=0$)
$$
AH(z,w)=|Az-Bw|^2+\Delta|w|^2
$$
shows that $-\Delta$ is a norm if and only if there are $z,w\in K$ not both zero with $H(z,w)=0$.
\end{proof}
Suppose $H$ is an indefinite integral binary Hermitian form of determinant $\Delta$ and $z=[a_0;a_1,\ldots]$ satisfies $H(z,1)=0$.  Define $H_n={}^{g_n^{-1}}H$, where
$$
g_n^{-1}=\left(\begin{array}{cc}0&1\\1&-a_n\\\end{array}\right)\cdots\left(\begin{array}{cc}0&1\\1&-a_0\\\end{array}\right), \ g_n=\left(
\begin{array}{cc}
p_n&p_{n-1}\\
q_n&q_{n-1}\\ 
\end{array}\right),
$$
with notation
$$
H_n=\left(\begin{array}{cc}A_n&-B_n\\-\overline{B_n}&C_n\\\end{array}\right).
$$
In particular, we have
\begin{align*}
A_n&=H(p_n,q_n),\\
C_n&=H(p_{n-1},q_{n-1})=A_{n-1}.
\end{align*}
Note that $H_n(1,z_n)=0$ for all $n\geq 1$ because $g_n(1/z_n)=z$.

The main observation for us is the following theorem, which is essentially Theorem 4.1 of \cite{BG} generalized to the other Euclidean $K$ and arbitrary integral binary Hermitian forms.  One could follow the inductive geometric proof of \cite{BG}, but we give an algebraic proof analogous to one demonstrating that real quadratic irrationals have eventually periodic simple continued fraction expansions, e.g. \cite{Kh}, Theorem 28.  In fact, we may as well note that the proof of Theorem \ref{thm:finite_orbit} below applies \textit{mutatis mutandis} to integral binary quadratic forms over $K$, showing that the continued fraction expansions of quadratic irrationals over $K$ are eventually periodic as expected.

\begin{theorem}\label{thm:finite_orbit}
If $[z:1]$ is a zero of an indefinite integral binary Hermitian form $H$, then the collection $\{H_n : n\geq0\}$ is finite.
\end{theorem}
\begin{proof}
In what follows, $\Delta=\Delta(H)=\Delta(H_n)$.  The inequality
$$
|z-p_n/q_n|\leq\frac{\kappa}{|q_n|^2}
$$
allows us to write
$$
p_n=q_nz+\frac{\gamma_n}{q_n}, \ |\gamma_n|\leq\kappa
$$
where $\kappa=\sup_{z,n}\{|q_n||q_nz-p_n|\}<2$ is the best constant from \cite{La} used in the proof of Lemma \ref{lemma:multconst}.

Substituting this into the formula for $A_n$ above gives
\begin{align*}
A_n&=H(q_nz+\gamma_n/q_n,q_n)=|q_n|^2H(z,1)+A\left(\overline{q_nz}\frac{\gamma_n}{q_n}+q_nz\frac{\overline{\gamma_n}}{\overline{q_n}}+\left|\frac{\gamma_n}{q_n}\right|^2\right)-\overline{B}\frac{\overline{q_n}}{q_n}\gamma_n-B\frac{q_n}{\overline{q_n}}\overline{\gamma_n}\\
&=A\left(\overline{q_nz}\frac{\gamma_n}{q_n}+q_nz\frac{\overline{\gamma_n}}{\overline{q_n}}+\left|\frac{\gamma_n}{q_n}\right|^2\right)-\overline{B}\frac{\overline{q_n}}{q_n}\gamma_n-B\frac{q_n}{\overline{q_n}}\overline{\gamma_n},
\end{align*}
and
\begin{align*}
|A_n|&\leq|A|\kappa^2+2|B|\kappa+2|A||z|\kappa\leq|A|\kappa^2+4|B|\kappa+2\kappa\sqrt{-\Delta}\\
&\leq\max\{|A|,|A|\kappa^2+4|B|\kappa+2\kappa\sqrt{-\Delta}\}=:\eta.
\end{align*}
(We take the max above so that the parameter $\eta$ is useful for bounds on $z_n$, $a_n$ when $n=0$, cf. Corollary \ref{cor:bounds} below.)  From this it follows that
\begin{align*}
|C_n|&=|A_{n-1}|\leq\eta,\\
|B_n|&=\sqrt{A_nC_n-\Delta}\leq\sqrt{\eta^2-\Delta},
\end{align*}
so that there are only finitely many possibilities for $H_n$.
\end{proof}

By requiring $H$ to be anisotropic, we bound the finitely many circles $Z(H_n)$ away from zero and infinity, obtaining bounded partial quotients.
\begin{cor}\label{cor:aniso}
If $[z:1]$ is a zero of an anisotropic indefinite integral binary Hermitian form $H$, then $z$ has bounded partial quotients in its nearest integer continued fraction expansion (and is therefore badly approximable over $K$).
\end{cor}
A quantitative measure of the ``hole'' around zero (see Figures \ref{fig:circgauss}, \ref{fig:circeisen}) can be given in terms of the determinant $\Delta=\det(H)$ of the form, which in turn bounds the partial quotients and controls the approximation constant $|z-p/q|\geq C'/|q|^2$.
\begin{cor}\label{cor:bounds}
If $z\in Z(H)$ is a zero of the anisotropic integral indefinite binary Hermitian form $H$ of determinant $\Delta$, then the remainders $z_n$, $n\geq0$, are bounded below by
$$
|z_n|\geq\frac{1}{\sqrt{-\Delta}+\sqrt{\eta^2-\Delta}},
$$
with partial quotients bounded above by
$$
|a_n|\leq\rho+\sqrt{-\Delta}+\sqrt{\eta^2-\Delta},
$$
where $\eta$ is as in the proof of Theorem \ref{thm:finite_orbit}.  From these bounds, one can produce a $C'(H)>0$ such that $|z-p/q|\geq C'/|q|^2$ for all $p/q\in K$ and $[z:1]\in Z(H)$.
\end{cor}
\begin{proof}
We use the notation of Theorem \ref{thm:finite_orbit}.  Since $1/z_n$ lies on $Z(H_n)\cap\mbb{C}_z$, which has radius $\sqrt{-\Delta}/|A_n|$ and center $B_n/A_n$, we have
\begin{align*}
\frac{1}{|z_n|}&\leq\left|\frac{B_n}{A_n}\right|+\frac{\sqrt{-\Delta}}{|A_n|}\leq\sqrt{\eta^2-\Delta}+\sqrt{-\Delta},\\
|z_n|&\geq\frac{1}{\sqrt{-\Delta}+\sqrt{\eta^2-\Delta}}.
\end{align*}
We also have
$$
a_{n+1}+z_{n+1}=\frac{1}{z_n}
$$
so that
$$
|a_{n+1}|\leq|z_{n+1}|+\frac{1}{|z_n|}\leq\rho+\sqrt{-\Delta}+\sqrt{\eta^2-\Delta}.
$$
From Theorem \ref{thm:bndpartquots}, it follows that $|z-p/q|\geq C'/|q|^2$ for all $p/q\in K$, where
$$
C'=\frac{1}{\alpha(\beta+1)(\beta+\rho+1)}, \ \beta=\rho+\sqrt{-\Delta}+\sqrt{\eta^2-\Delta}.
$$
\end{proof}
We conclude this section with an example of the above over $\mbb{Q}(i)$.  The number
$$
z=\sqrt{3}e^{2\pi i/5}=[1+2i; -1+i,  -3, 2+2i, -1+3i, -2, -2i, 2+2i, 3-i, -2+2i, \ldots]
$$
is a zero of the anisotropic indefinite integral binary Hermitian form $H(z,w)=|z|^2-3|w|^2$.  Data from 10,000 iterations of the continued fraction algorithm and bounds from Corollary \ref{cor:bounds} are:
\begin{align*}
\max_{0\leq n<10000}\{|a_n|\}&=4.47213\ldots\leq 7.22749\ldots,\\ 
\min_{0\leq n<10000}\{|z_n|\}&=0.25201\ldots\geq0.15336\ldots, \\
\min_{0\leq n<10000}\{|q_n(q_nz-p_n)|\}&=0.28867\ldots\geq 0.00563\ldots.
\end{align*}
There are 64 distinct partial quotients $a_n$ and 56 distinct forms $H_n$.  The remainders $z_n$ are shown on the left in Figure \ref{fig:algexamples} below.  The remainders appear to have different densities along the arcs $Z(H_n)\cap V$, spending more time on the circles of radius $\sqrt{3}$ (show in black in Figure \ref{fig:algexamples}) than on the circles of radius $\sqrt{3}/2$ (shown in red in Figure \ref{fig:algexamples}).  The bound on the normalized error seems to be rather poor, but it is interesting to note that the minimum $0.28867\ldots$ above agrees with $\frac{1}{2\sqrt{3}}$ to high precision.
\begin{figure}[hbt]
\begin{center}
\includegraphics[scale=0.4]{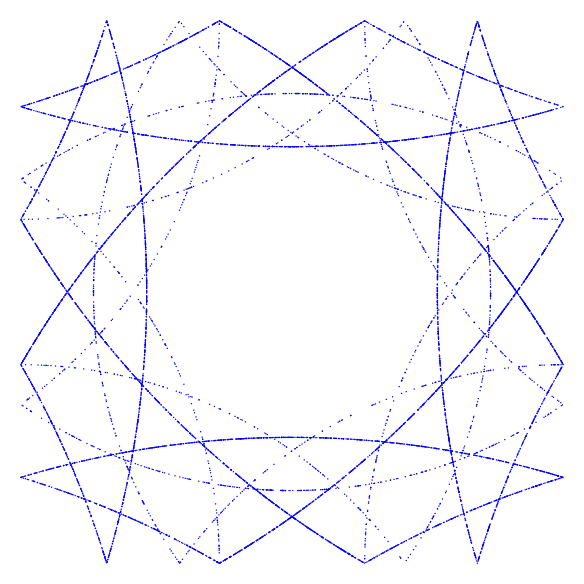}
\includegraphics[scale=0.4]{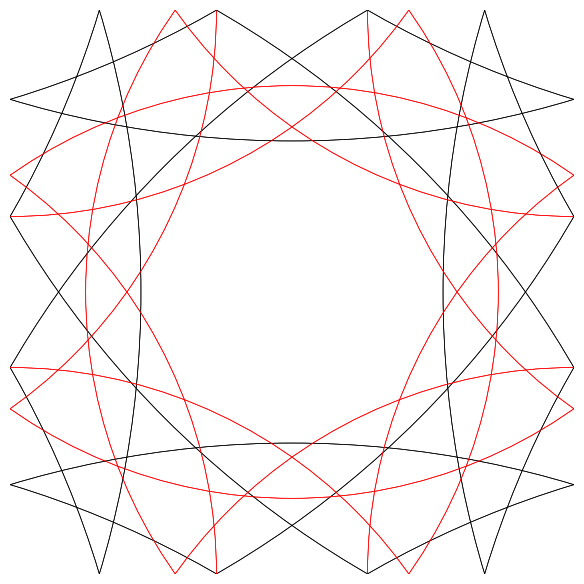}
\caption{The first 10,000 remainders $z_n$ of $z=\sqrt{3}e^{2\pi i/5}$ over $\mbb{Q}(i)$ (left) and the arcs $Z(H_n)\cap V$ (right).}
\label{fig:algexamples}
\end{center}
\end{figure}
\section*{Badly approximable circles over arbitrary imaginary quadratic fields via the Dani correspondence}

The fact that all of the points on a $K$-rational circle without rational points are badly approximable holds for \textit{any} imaginary quadratic $K$, not only those that are Euclidean.  We no longer have continued fractions at our disposal, but we have a characterization of badly approximable $z$ obtained from consideration of a geodesic trajectory in $SL_2(\mc{O})\backslash\mbb{H}^3$ ``aimed'' at $z$, a variation of a result of Dani (\cite{Da3}).  This is a justification for the intuitive equivalence of ``badly approximable'' and ``staying out of the cusps'' (or ``bounded partial quotients'' from continued fractions).

If $K$ has class number $h(K)$, then $SL_2(\mc{O})\backslash\mbb{H}^3$ has $h$ cusps, i.e. there are $h$ orbits for the action of $SL_2(\mc{O})$ on $P^1(K)$ (\cite{EGM} Chapter 7, Theorem 2.4).  It therefore makes sense to discuss approximating a complex number with rationals representing different ideal classes.  However, we will continue to use the notion of badly approximable in the form
$$
\text{there exists } C'>0 \text{ such that for all } p/q\in K, \ |z-p/q|\geq C'/|q|^2.
$$

We begin with a brief discussion of the spaces under consideration.  The group $SL_2(\mbb{C})$ acts transitively on three dimensional hyperbolic space $\mbb{H}^3$ (as a subset of the real quaternions) via fractional linear transformations:
$$
\mbb{H}^3=\{\zeta=z+tj : z\in\mbb{C}, \ 0<t\in\mbb{R}\}, \ g\cdot \zeta=(a\zeta+b)(c\zeta+d)^{-1}.
$$
The stabilizer of $j\in\mbb{H}^3$ is $SU_2(\mbb{C})$, double covering $SO_3(\mbb{R})\cong SU_2(\mbb{C})/\{\pm1\}$.  Using the basepoint $j\in\mbb{H}^3$ (and some choice of frame at $j$, of which we will have no need), we make the following identifications:
\begin{align*}
&SL_2(\mbb{C})/SU_2(\mbb{C})&\mbb{H}^3,\\
&PSL_2(\mbb{C})&\text{oriented orthonormal frame bundle of } \mbb{H}^3,\\
&SL_2(\mc{O})\backslash SL_2(\mbb{C})/SU_2(\mbb{C})&\text{the Bianchi orbifold } SL_2(\mc{O})\backslash\mbb{H}^3,\\
&SL_2(\mc{O})\backslash SL_2(\mbb{C})&\text{oriented orthonormal frame bundle of } SL_2(\mc{O})\backslash\mbb{H}^3.\\
\end{align*}
For example,
$$
\left(
\begin{array}{cc}
1&z\\
0&1\\
\end{array}
\right) 
\left(
\begin{array}{cc}
e^{t/2}&0\\
0&e^{-t/2}\\
\end{array}
\right)SU_2(\mbb{C}) \ \leftrightarrow \ z+e^tj.
$$
We also note that the projection $SL_2(\mc{O})\backslash SL_2(\mbb{C})\to SL_2(\mc{O})\backslash\mbb{H}^3$ is proper and closed.

The following theorem characterizes badly approximable numbers by the boundedness of a (framed) geodesic trajectory.
\begin{theorem}[Dani correspondence]\label{thm:dani}
For $z\in\mbb{C}$, define
$$
\Omega_z=\left\{
\left(
\begin{array}{cc}
1&z\\
0&1\\
\end{array}
\right) 
\left(
\begin{array}{cc}
e^{-t}&0\\
0&e^t\\
\end{array}
\right) : t\geq0
\right\}\subseteq SL_2(\mbb{C}).
$$
The trajectories
\begin{align*}
&SL_2(\mc{O})\cdot\Omega_z\subseteq SL_2(\mc{O})\backslash SL_2(\mbb{C}),\\
\omega_z:=&SL_2(\mc{O})\cdot\Omega_z\cdot SU_2(\mbb{C})\subseteq SL_2(\mc{O})\backslash\mbb{H}^3,
\end{align*}
are precompact if and only if $z$ is badly approximable.
\end{theorem}
The Dani correspondence as stated above follows from a version of Mahler's compactness criterion.
\begin{theorem}[Mahler's compactness criterion]\label{mahler}
Let $\Omega\subseteq SL_2(\mbb{C})$.  The set $SL_2(\mc{O})\cdot\Omega\subseteq SL_2(\mc{O})\backslash SL_2(\mbb{C})$ is precompact if and only if
$$
\inf\{\|Xg\|_2 : g\in\Omega, \ X=(x_1,x_2)\in\mc{O}^2\setminus\{(0,0)\}\}>0,
$$
where $\|X\|_2=\sqrt{|x_1|^2+|x_2|^2}$.
\end{theorem}
For completeness, we give a proof of Theorem \ref{mahler} at the end of the section.  We now give a proof of the Dani correspondence as stated in Theorem \ref{thm:dani}.
\begin{proof}[Proof of Dani correspondence]
For the proof, note that
$$
\mc{O}^2\cdot\Omega_z=\{(e^{-t}q,e^t(p+qz)) :t\geq0, \  (q,p)\in\mc{O}^2\}.
$$
Suppose $z$ is badly approximable with $|q(qz+p)|\geq C'$ for all $p/q\in K$.  If there exists $t\geq0$ and $p/q\in K$ with $\|(e^{-t}q,e^t(qz+p))\|_2<\sqrt{C'}$, then taking the product of the coordinates gives
$$
|e^{-t}qe^t(qz+p)|=|q(qz+p)|<C',
$$
a contradiction.  Hence $\inf\{\|Xg\|_2 : g\in\Omega_z, \ X\in\mc{O}^2\setminus\{(0,0)\}\}\geq\sqrt{C'}$ and $SL_2(\mc{O})\cdot\Omega_z$ is precompact by Mahler's criterion.

If $z$ is not badly approximable, then for every $n>0$ there exists $p_n/q_n\in K$ such that $|q_n(q_nz+p_n)|<1/n^2$.  If $t_n$ is such that $e^{-t_n}|q_n|=1/n$, then $|e^{t_n}(q_nz+p_n)|=n|q_n(q_nz+p_n)|<1/n$ and $\|(e^{-t_n}q_n,e^{t_n}(q_nz+p_n))\|_2\leq\sqrt{2}/n$.  Therefore $SL_2(\mc{O})\cdot\Omega_z$ is not contained in any compact set by Mahler's criterion.  The result for $\omega_z$ follows from the invariance $\|Xk\|_2=\|X\|_2$ for $X\in\mbb{C}^2$, $k\in SU_2(\mbb{C})$, and the fact that the projection $SL_2(\mc{O})\backslash SL_2(\mbb{C})\to SL_2(\mc{O})\backslash\mbb{H}^3$ is proper and closed.
\end{proof}

One obvious way for the trajectory $\omega_z$ to remain bounded is for it to be asymptotic to a compact object, such as a compact geodesic surface.  The following lemma provides a plethora of compact geodesic surfaces in the Bianchi orbifolds $SL_2(\mc{O})\backslash\mbb{H}^3$.

%
\begin{lemma}[\cite{MR}, \S9.6]\label{lemma:cmpct}
Let $H$ be an integral indefinite binary Hermitian form.  The orientation preserving stabilizer of the zero set $Z(H)$,
$$
\text{Stab}^+(Z(H))=\left\{g\in SL_2(\mc{O}) :
\begin{array}{c}
g\cdot Z(H)=Z(H) \text{ and }\\
g \text{ preserves the components of } P^1(\mbb{C})\setminus Z(H)
\end{array}\right\},
$$
is a maximal non-elementary Fuchsian subgroup of $SL_2(\mc{O})$.  If $H$ is anisotropic, then $\text{Stab}^+(Z(H))$ is cocompact.  In particular, if $P(H)\subseteq\mbb{H}^3$ is the geodesic plane with boundary $Z(H)$, then the image of $P(H)$ in $SL_2(\mc{O})\backslash\mbb{H}^3$ is compact for anisotropic $H$.
\end{lemma}
Together, Lemma \ref{lemma:cmpct} and Theorem \ref{thm:dani} imply the following.

\begin{theorem}\label{thm:badapproxcirc}
Let $K$ be any imaginary quadratic field and $H$ an anisotropic indefinite $K$-rational binary Hermitian form, i.e.
$$
H=Az\bar{z}-B\bar{z}w-\overline{B}z\bar{w}+Cw\bar{w}, \ A,C\in\mbb{Q}, \ B\in K, \ 0<-\Delta(H)=B\overline{B}-AC\not\in N^K_{\mbb{Q}}(K).
$$
If $z\in\mbb{C}$, $H(z,1)=0$, then $z$ is badly approximable over $K$, i.e. there exists $C'>0$ such that for all $p/q\in K$, we have $|z-p/q|\geq C'/|q|^2$.
\end{theorem}
\begin{proof}
Let $\pi$ denote the projection $\pi:\mbb{H}^3\to SL_2(\mc{O})\backslash\mbb{H}^3$.  The circle $Z(H)$ is the boundary at infinity of a geodesic plane $P(H)$ (a hemisphere orthogonal to $\mbb{C}\subseteq\partial\mbb{H}^3$).  The geodesic ray
$$
\Omega_z\cdot SU_2(\mbb{C})=\{z+e^{-2t}j : 0\leq t<\infty\}
$$
is asymptotic to $P(H)$ as they share the point $z$ at infinity.  Since $\pi$ does not increase distances and $\pi(P(H))$ is compact by Lemma \ref{lemma:cmpct}, $\pi(\Omega_z\cdot SU_2(\mbb{C}))=\omega_z$ is bounded in $SL_2(\mc{O})\backslash \mbb{H}^3$.  By  Theorem \ref{thm:dani}, $z$ is badly appoximable.
\end{proof}


We now give a proof of Mahler's criterion Theorem \ref{mahler} in this setting, reducing it to the following standard version over $\mbb{Z}$.

\begin{theorem}[\cite{M}, Theorem 2]\label{thm:mahlerorig}
Let $\Omega\subseteq GL_n(\mbb{R})$.  Then $GL_n(\mbb{Z})\cdot\Omega\subseteq GL_n(\mbb{Z})\backslash GL_n(\mbb{R})$ is precompact if and only if the following two conditions are satisfied:
$$
\sup\{|\det(g)| : g\in \Omega\}<\infty, \ \inf\{\|Xg\|_2 : g\in\Omega, X\in\mbb{Z}^n\}>0,
$$
where $\|X\|_2=\sqrt{x_1^2+\ldots+x_n^2}$ for $X=(x_1,\ldots,x_n)\in\mbb{R}^n$.  In other words, the lattices generated by the rows of elements of $\Omega$ must have bounded covolume and no arbitrarily short vectors.
\end{theorem}
\begin{proof}[Proof of Mahler's criterion]
Choose an integral basis for $\mc{O}$, say $1$ and $\omega=\frac{D_K+\sqrt{D_K}}{2}$ for concreteness, where $D_K$ is the field disciminant,
$$
D_K=\left\{
\begin{array}{cc}
-d,&d\equiv3\bmod4\\
-4d,&d\equiv 1, 2\bmod4\\
\end{array}
\right..
$$
We have a homomorphism
$$
\phi:\mbb{C}\to M_2(\mbb{R}), \ z\mapsto\left(\begin{array}{cc}
r&s\\
s\frac{D_K(1-D_K)}{4}&r+sD_K\\
\end{array}
\right), \ z=r+s\omega,
$$
taking a complex number $z$ to the matrix of multiplication by $z$ in the our integral basis.  This extends to a homomorphism
$$
\Phi:SL_2(\mbb{C})\to SL_4(\mbb{R}), \ 
\left(\begin{array}{cc}
z_1&z_2\\
w_1&w_2\\
\end{array}\right)\mapsto
\left(\begin{array}{cc}
\phi(z_1)&\phi(z_2)\\
\phi(w_1)&\phi(w_2)\\
\end{array}\right),
$$
with $\Phi(SL_2(\mbb{C}))\cap SL_4(\mbb{Z})=\Phi(SL_2(\mc{O}))$.  Hence we obtain a closed embedding
$$
\widetilde{\Phi}:SL_2(\mc{O})\backslash SL_2(\mbb{C})\to SL_4(\mbb{Z})\backslash SL_4(\mbb{R}).
$$
One can easily verify that the bijection
$$
\Psi:\mbb{C}^2\to\mbb{R}^4, \ (a+b\omega,c+d\omega)\mapsto(a,b,c,d)
$$
is $SL_2(\mbb{C})$-equivariant, i.e.
$$
\Psi((a+b\omega,c+d\omega)g)=(a,b,c,d)\Phi(g), \ g\in SL_2(\mbb{C}),
$$
and that the norms $\|\Psi(\cdot)\|_2$, $\|\cdot\|_2$ are equivalent on $\mbb{C}^2$:
$$
R_+\|\Psi(\cdot)\|_2\leq\|\cdot\|_2\leq R_-\|\Psi(\cdot)\|_2,
$$
where
$$
R_{\pm}=\left(\frac{2}{1+|\omega|^2\pm|1+\omega^2|}\right)^{1/2}
$$
are the radii of the inscribed and circumscribed circles of the ellipse $|a+b\omega|^2=1$.  Applying the standard version of Mahler's criterion (Theorem \ref{thm:mahlerorig} above) to $\Phi(\Omega)$ gives the result.
\end{proof}

For a generalization of Mahler's criterion and the Dani correspondence tailored to simultaneous approximation over number fields which includes what is presented here, see \cite{EGL}.
\section*{Examples of badly approximable \textit{algebraic} numbers over any imaginary quadratic field}
In this section, we emphasize the fact that there are many \textit{algebraic} numbers satisfying the hypotheses of Corollary \ref{cor:aniso} and Theorem \ref{thm:badapproxcirc} and give a characterization of these numbers.

For $z$ such that $|z|^2=s/t\in\mbb{Q}$ is not a norm from $K$, the anisotropic integral form $\left(\begin{array}{cc}t&0\\0&-s\end{array}\right)$ shows that $z$ is badly approximable.  This essentially exhausts all of the examples we've given as we can translate an integral form by a rational to center it at zero, then clear denominators to obtain an integral form as described, i.e.
$$
H=\left(\begin{array}{cc}
A&-B\\
-\overline{B}&C\\
\end{array}\right),\ 
g=\left(\begin{array}{cc}
1&-B/A\\
0&1\\
\end{array}\right), \ A\cdot{}^gH=\left(\begin{array}{cc}
A^2&0\\
0&AC-B\overline{B}\\
\end{array}\right).
$$

For some specific algebraic examples, consider quadratically scaled roots of unity $z=\sqrt{n}\zeta$ for $|z|^2=n\in\mbb{Q}$ not a norm, or the generalizations of examples from \cite{BG}, $z=\sqrt[m]{a}+\sqrt{\sqrt[m]{a^2}-n}$ for $a\in\mbb{Q}$, $\sqrt[m]{a^2}<n$, and $n=|z|^2$ not a norm.  See Figures \ref{fig:circgauss}, \ref{fig:circeisen} for visualizations of the orbits of algebraic numbers satisfying $|z|^2=n$ for various $n$ and $d=1,3$.

\begin{figure}[hbt]
\begin{center}
\includegraphics[scale=.25]{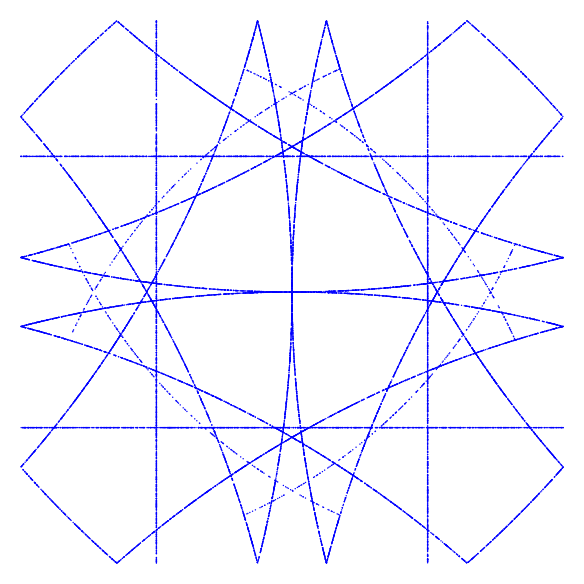}
\includegraphics[scale=.25]{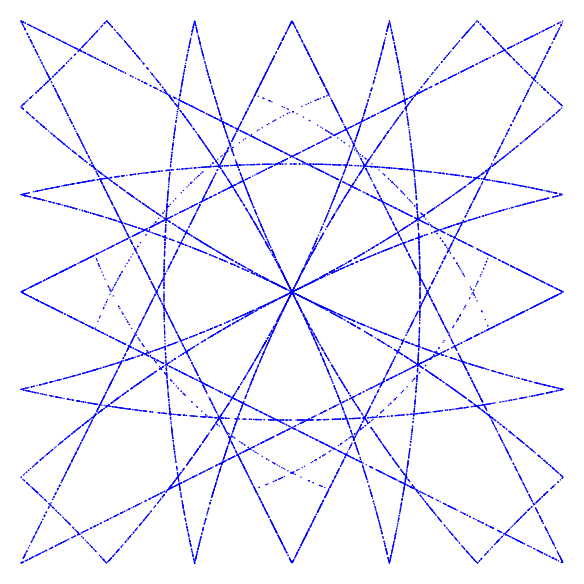}
\includegraphics[scale=.25]{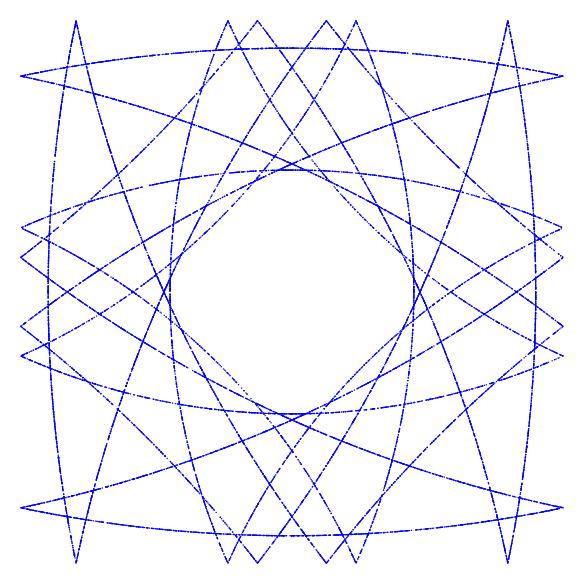}
\includegraphics[scale=.25]{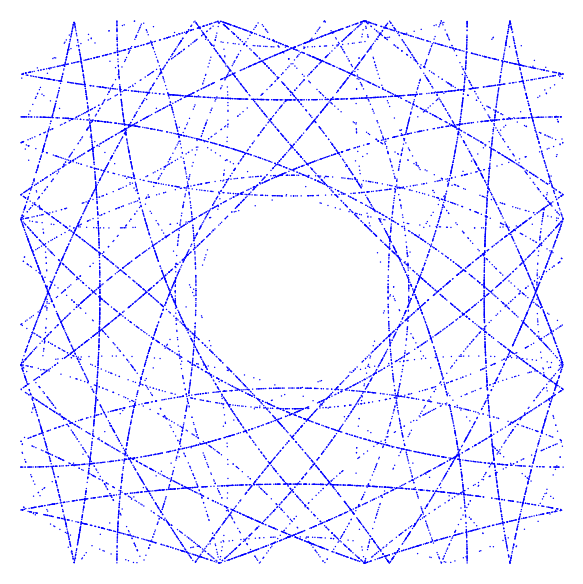}
\caption{20,000 iterates of $T$ for $z=\sqrt[3]{2}+\sqrt{\sqrt[3]{4}-n}$ over $\mbb{Q}(\sqrt{-1})$ with $n=4,5,6,7$.}
\label{fig:circgauss}
\end{center}
\end{figure}

\begin{figure}[hbt]
\begin{center}
\includegraphics[scale=.25]{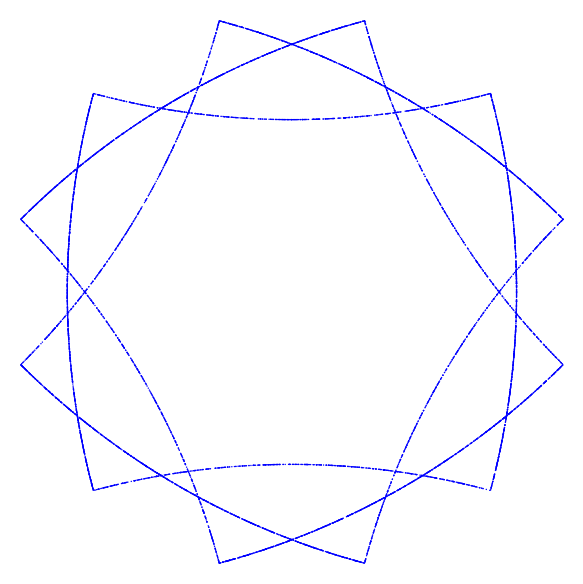}
\includegraphics[scale=.25]{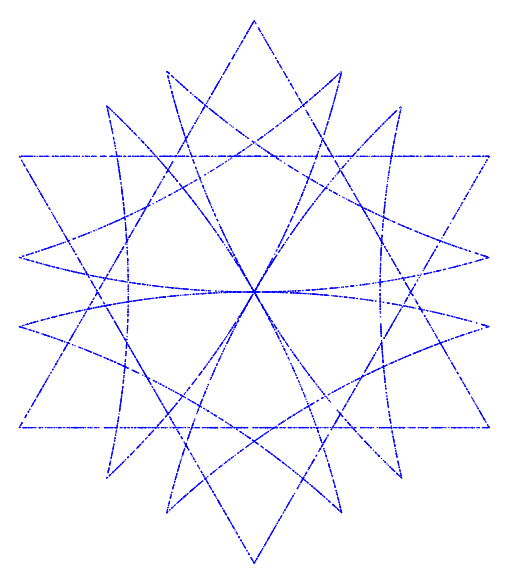}
\includegraphics[scale=.25]{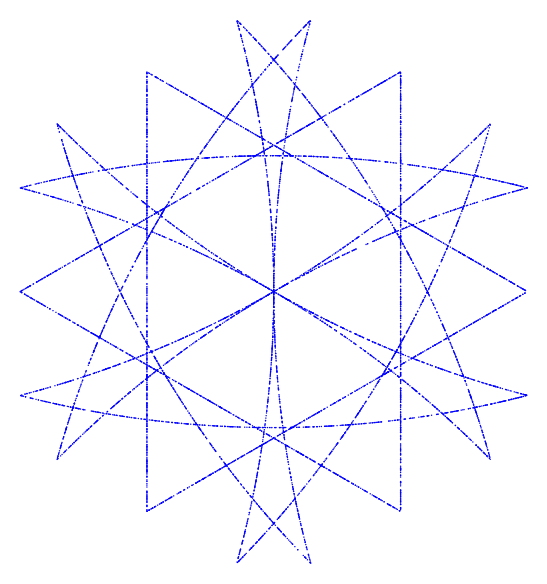}
\includegraphics[scale=.25]{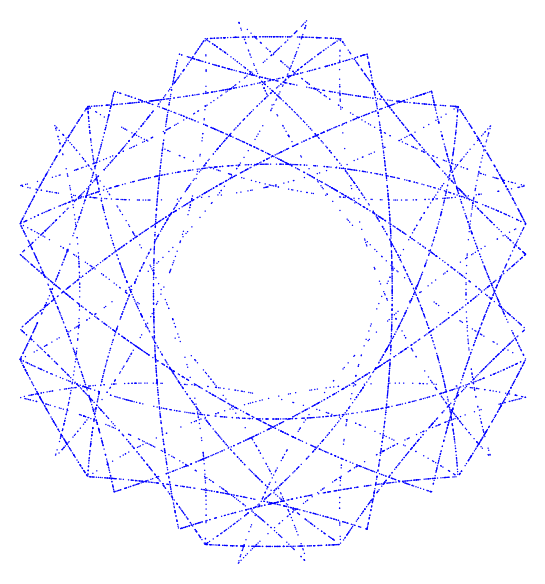}
\caption{10,000 iterates of $T$ for $z=\sqrt[3]{2}+\sqrt{\sqrt[3]{4}-n}$ over $\mbb{Q}(\sqrt{-3})$ with $n=2,3,4,5$.}
\label{fig:circeisen}
\end{center}
\end{figure}

We would like to characterize the badly approximable algebraic numbers captured above.  One such characterization comes from a parameterization of the algebraic numbers on the unit circle (taken from the mathoverflow post \cite{KCd}).

\begin{lemma}[\cite{KCd}]\label{lemma:unit_circ}
The algebraic numbers $w$ on the unit circle, $w\overline{w}=1$, are those numbers of the form
$$
w=\frac{u\pm\sqrt{u^2-4}}{2}
$$
where $u$ is a real algebraic number in the interval $[-2,2]$.  If $u\neq\pm2$, the minimal polynomial $f$ of $w$ is
$$
f(t)=t^{m}g(t+1/t)
$$
where $g(t)$ is the minimal polynomial of $u$, $\deg(g)=m$.  In particular, the degree of $f$ is even.
\end{lemma}
\begin{proof}
We know that $w, \overline{w}=1/w$ have the same minimal polynomial $f(x)\in\mbb{Q}[x]$, say of degree $d$.  One can deduce that $x^df(1/x)=f(x)$ so that $f(x)$ is palindromic (if $f(x)=\sum_kf_kx^k$ then $f_{d-k}=f_k$) and since the roots come in reciprocal pairs, $d$ is even.  The even degree palindromic polynomials are of the form $f(x)=x^{d/2}g(x+1/x)$ for some polynomial $g$
\begin{align*}
f(x)&=\sum_{k=0}^df_kx^k=x^{d/2}\sum_{k=0}^{d/2}f_k(x^k+1/x^k)=x^{d/2}\left(r(x+1/x)+\sum_{k=0}^{d/2}f_k(x+1/x)^k\right)\\
&=x^{d/2}g(x+1/x), \ g(x)=r(x)+\sum_{k=0}^{d/2}f_kx^k,
\end{align*}
noting that the difference $(x+1/x)^k-(x^k+1/x^k)$ is a polynomial in $x+1/x$ by symmetry of the binomial coefficients.  The roots of the even degree palindromic $f$ on the circle double cover the roots of $g$ in the interval $(-2,2)$ (via $w=e^{i\theta}\mapsto u=2\cos\theta$).  Conversely, taking an irreducible polynomial $g(x)\in\mbb{Q}[x]$ of degree $d/2$ with a root in the interval $(-2,2)$ gives a degree $d$ irreducible polynomial $f(x)=x^{d/2}g(x+1/x)$ with roots on the unit circle.  
\end{proof}

Hence we have the following corollary describing the badly approximable algebraic numbers coming from Corollary \ref{cor:aniso} and Theorem \ref{thm:badapproxcirc}.

\begin{cor}\label{cor:algebraic}
For any imaginary quadratic field $K$, there are algebraic numbers of each even degree over $\mbb{Q}$ that are badly approximable over $K$.  Specifically, for any real algebraic number $u\in[-2,2]$, any positive $n\in\mbb{Q}\setminus N^K_{\mbb{Q}}(K)$, and any $t\in K$, the number
$$
z=t+\sqrt{n}\cdot\frac{u\pm\sqrt{u^2-4}}{2}
$$
is badly approximable.
\end{cor}

For instance, the examples in Figures \ref{fig:circgauss} and \ref{fig:circeisen} have $t=0$ and $u=2^{4/3}/\sqrt{n}$.
\section*{Appendix: Monotonicity of denominators for $d=7,11$}
The purpose of this appendix is to prove Proposition \ref{prop:mono} for $d=7,11$.  Proofs for $d=1,3$ are found in \cite{Hur1} and a proof for $d=2$ in \cite{Lu} ($\S VII$, Satz 11).  Monotonicity for $d=7,11$ was stated in \cite{La} without proof (for reasons obvious to anyone reading what follows).  All of the proofs follow the same basic outline, with $d=11$ being the most tedious.
\begin{proof}[Proof of Proposition \ref{prop:mono}.]
For the purposes of this proof, define $k_n=q_n/q_{n-1}$; we will show $|k_n|>1$.  We will also use the notation $B_t(r)$ for the ball of radius $t$ centered at $r\in\mc{O}$.  When $n=1$ we have $|k_1|=|a_1|\geq1/\rho>1$.  The recurrence $k_n=a_n+1/k_{n-1}$ is immediate from the definitions.  We argue by contradiction following \cite{Hur1}.  Suppose $n>1$ is the smallest value for which $|k_n|\leq1$ so that $|k_i|>1$ for $1\leq i<n$.  Then
$$
|a_n|=\left|k_n-1/k_{n-1}\right|<2.
$$
For small values of $a_i$, those for which $(a_i+V)\cap V^{-1}\cap(\mbb{C}\setminus V^{-1})\neq\emptyset$, the values of $a_{i+1}$ are restricted (cf. Figure \ref{fig:basediag}).  More generally, there are arbitrarily long ``forbidden sequences'' stemming from these small values of $a_i$.  We will use some of the forbidden sequences that arise in this fashion to show that the assumption $k_n<1$ leads to a contradiction.
\begin{itemize}
\item{($d=7$)}
The only allowed values of $a_n$ for which $|a_n|<2$ are $a_n=\frac{\pm1\pm\sqrt{-7}}{2}$.  By symmetry, we suppose $a_n=\frac{1+\sqrt{-7}}{2}=:\omega$ without loss of generality.  It follows that $k_n\in B_1(\omega)\cap B_1(0)$.  Subtracting $\omega=a_n$, we see that $1/k_{n-1}\in B_1(0)\cap B_1(-\omega)$.  Applying $1/z$ then gives $k_{n-1}\in B_1(\omega-1)\setminus B_1(0)$.  Since $k_{n-1}=a_{n-1}+1/k_{n-2}$, the only possible values for $a_{n-1}$ are $\omega$, $\omega-1$, $\omega-2$, $2\omega-1$, and $2\omega-2$.  One can verify that the two-term sequences
\begin{center}
\begin{tabular}{ ||c|c|| } 
\hline
$a_i$&$a_{i+1}$\\
\hline
\hline
$\omega-2$&$\omega$\\
\hline
$\omega-1$&$\omega$\\
\hline
$\omega$&$\omega$\\
\hline
\end{tabular}
\end{center}
are forbidden, so that $a_{n-1}=2\omega-1$ or $2\omega-2$.  We now have either $k_{n-1}\in B_1(2\omega-1)\cap B_1(\omega-1)$ if $a_{n-1}=2\omega-1$, or $k_{n-1}\in B_1(2\omega-2)\cap B_1(\omega-1)$ if $a_{n-1}=2\omega-2$.  Subtracting $a_{n-1}$ and applying $1/z$ shows that $a_{n-2}=2\omega-1$ or $2\omega-2$ if $a_{n-1}=2\omega-1$, and $a_{n-2}=2\omega$ or $2\omega-1$ if $a_{n-1}=2\omega-2$.  

Continuing shows that for $i\leq n-1$
$$
k_i\in\left(B_1(2\omega-2)\cup B_1(2\omega-1)\cup B_1(2\omega)\right)\cap\left(B_1(\omega-1)\cup B_1(\omega)\right),
$$
the green region on the left in Figure \ref{fig:contra}.  This is impossible; for instance $k_1=a_1\in\mc{O}$ but the region above contains no integers.
\item{($d=11$)}
The only allowed values of $a_n$ for which $|a_n|<2$ are $\frac{\pm1\pm\sqrt{-11}}{2}$.  By symmetry, we suppose $a_n=\frac{1+\sqrt{-11}}{2}=:\omega$ without loss of generality.  Hence $k_n\in B_1(0)\cap B_1(\omega)$.  Subtracting $a_n$ and applying $1/z$ shows that $k_{n-1}\in B_{1/2}(\frac{\omega-1}{2})\setminus B_1(0)$ and $a_{n-1}=\omega-1$ or $\omega$.  If $a_{n-1}=\omega-1$, subtracting, applying $1/z$ and using the three-term forbidden sequences
\begin{center}
\begin{tabular}{ ||c|c|c|| } 
\hline
$a_i$&$a_{i+1}$&$a_{i+2}$\\
\hline
\hline
$\omega-1$&$\omega-1$&$\omega$\\
\hline
$\omega$&$\omega-1$&$\omega$\\
\hline
$\omega+1$&$\omega-1$&$\omega$\\
\hline
\end{tabular}
\end{center}
shows that $a_{n-2}=2\omega$ or $2\omega-1$, with $k_{n-2}\in B_1(2\omega)\cap B_1(\omega)$ or $B_1(2\omega-1)\cap B_1(\omega)$ respectively.

If $a_{n-1}=\omega$, subtracting and applying $1/z$ gives $a_{n-2}=\omega-2$ or $\omega-1$ with $k_{n-2}\in B_{1/2}(\frac{\omega-2}{2})\cap B_1(\omega-2)$ or $(B_{1/2}(\frac{\omega-2}{2})\cap B_1(\omega-1))\setminus B_{1/2}(\frac{\omega-1}{2})$ respectively.  If $a_{n-2}=\omega-1$, subtracting $a_{n-2}$, applying $1/z$, and using the three-term forbidden sequences
\begin{center}
\begin{tabular}{ ||c|c|c|| } 
\hline
$a_i$&$a_{i+1}$&$a_{i+2}$\\
\hline
\hline
$\omega-1$&$\omega-1$&$\omega$\\
\hline
$\omega-2$&$\omega-1$&$\omega$\\
\hline
\end{tabular}
\end{center}
shows that $a_{n-3}=2\omega-2$ or $2\omega-1$ with $k_{n-3}\in B_1(2\omega-2)\cap B_1(\omega-1)$ or $B_1(2\omega-1)\cap B_1(\omega-1)$ respectively.  If $a_{n-2}=\omega-2$, subtracting $a_{n-2}$ and applying $1/z$ shows that $a_{n-3}=\omega$ or $\omega+1$, with $k_{n-3}\in(B_1(\omega)\cap B_{1/2}(\frac{\omega+1}{2}))\setminus B_1(0)$ or $B_1(\omega+1)\cap B_{1/2}(\frac{\omega+1}{2})$ respectively.  If $a_{n-3}=\omega+1$, we loop back into a symmetric version of a case previously considered (namely $k_{n-4}\in B_{1/2}(\frac{\omega-2}{2})\setminus B_1(0)$ and $a_{n-4}=\omega-1$ or $\omega-2$).  If $a_{n-3}=\omega$, subtracting $a_{n-3}$, applying $1/z$, and using the forbidden sequences
\begin{center}
\begin{tabular}{ ||c|c|c|c|c|| } 
\hline
$a_i$&$a_{i+1}$&$a_{i+2}$&$a_{i+3}$&$a_{i+4}$\\
\hline
\hline
$\omega-1$&$\omega$&$\omega-2$&&\\
\hline
$\omega$&$\omega$&$\omega-2$&$\omega$&\\
\hline
$\omega+1$&$\omega$&$\omega-2$&$\omega$&$\omega$\\
\hline
\end{tabular}
\end{center}
shows that $a_{n-4}=2\omega$ or $2\omega-1$ with $k_{n-4}\in B_1(2\omega)\cap B_1(\omega)$ or $B_1(2\omega-1)\cap B_1(\omega)$ resepectively.

Continuing, we find $k_i$ for $i\leq n-1$ restricted to the region depicted on the right in Figure \ref{fig:contra}.  This region contains no integers, contradicting $k_1\in\mc{O}$.
\end{itemize}
\begin{figure}[hbt]
\begin{center}
\includegraphics[scale=.5]{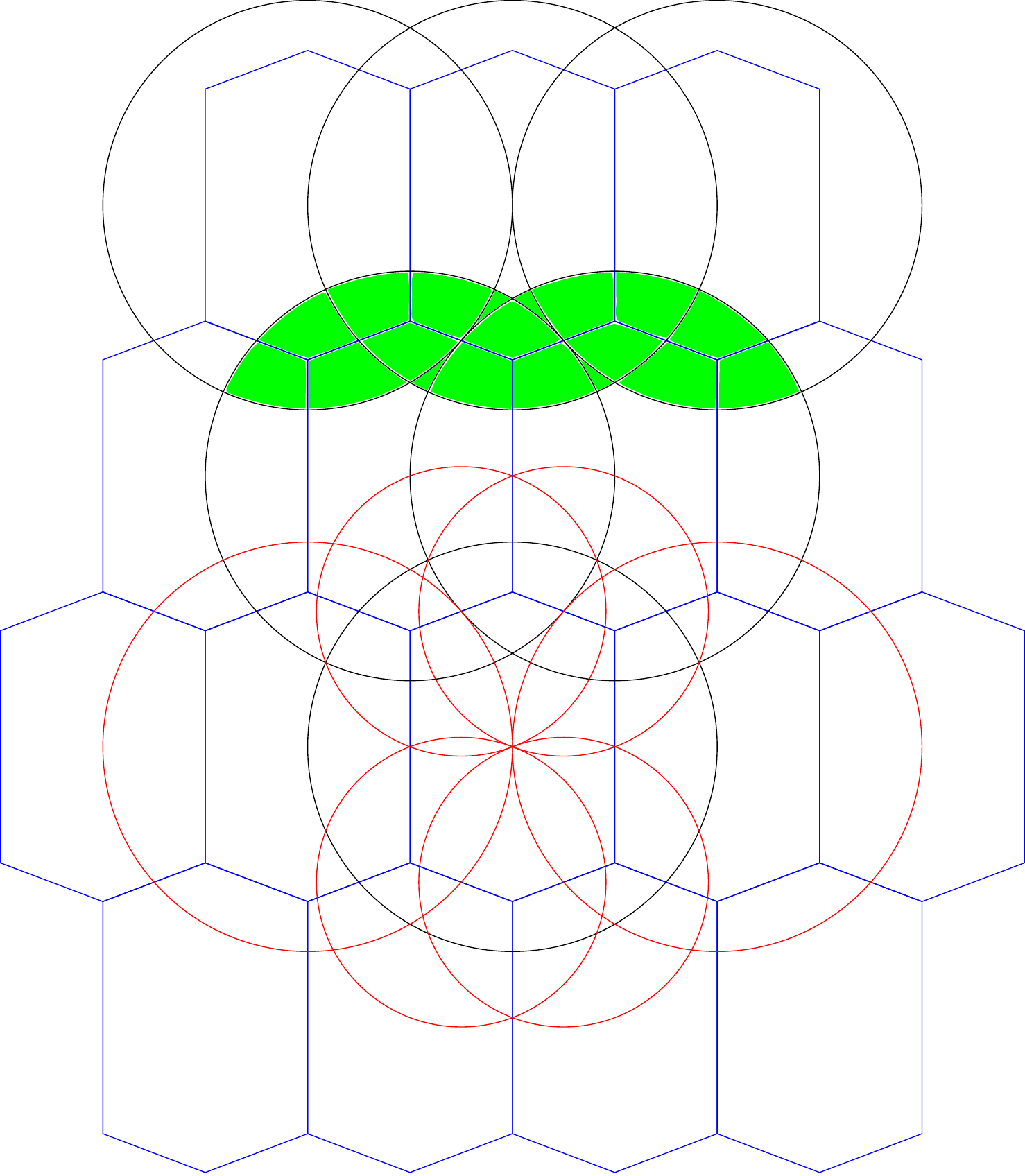}
\includegraphics[scale=.5]{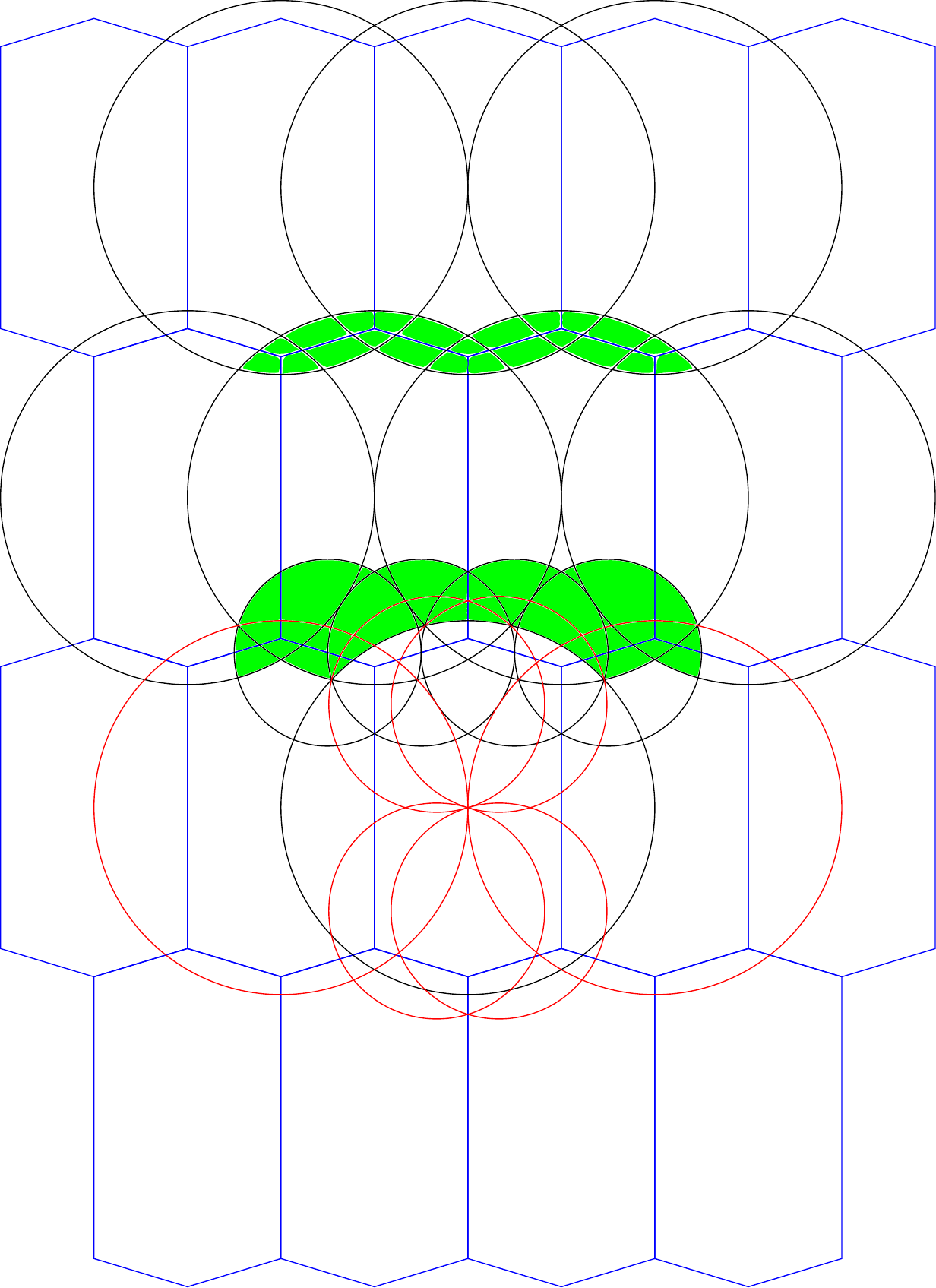}
\caption{In the proof of Proposition \ref{prop:mono}, the assumption $k_n<1$ with $a_n=\frac{1+\sqrt{-7}}{2}$ (left) or $a_n=\frac{1+\sqrt{-11}}{2}$ (right) leads to restricted potential values for $k_i$ with $i<n$ (green regions).}
\label{fig:contra}
\end{center}
\end{figure}
\end{proof}
\FloatBarrier

\end{document}